\documentclass[a4paper,11pt, reqno]{amsart}
\usepackage[utf8]{inputenc}
\usepackage{amsmath}
\usepackage{amssymb}
\usepackage{amsthm}
\usepackage{amsrefs} 
\usepackage{a4wide}
\usepackage{paralist}
\usepackage{marvosym}
\usepackage{esint}
\usepackage{dsfont}
\usepackage{extarrows}
\usepackage{mathrsfs}
\usepackage{pdfpages}
\usepackage{blindtext}

\newtheorem{theorem}{Theorem}[section]
\newtheorem{lemma}{Lemma}[section]
\newtheorem{proposition}{Proposition}[section]

\theoremstyle{definition}
\newtheorem{definition}{Definition}[section]
\newtheorem*{acknowledgements}{Acknowledgements}

\theoremstyle{remark}
\newtheorem{remark}{Remark}[section]

\DeclareMathOperator{\RR}{\mathbb{R}}
\DeclareMathOperator{\Sp}{\mathbb{S}}
\def\PP{\mathbb{P}}
\def\EE{\mathbb{E}}

\def\cA{\mathcal{A}}
\def\cB{\mathcal{B}}

\def\cS{\mathcal{S}}

\def\bZ{\mathbf{Z}}
\def\bV{\mathbf{V}}
\def\bE{\mathbf{E}}
\def\bM{\mathbf{M}}
\def\bS{\mathbf{S}}
\def\bX{\mathbf{X}}
\def\bY{\mathbf{Y}}

\def\sG{\mathscr{G}}
\def\sP{\mathscr{P}}
\def\sT{\mathscr{T}}

\newcommand{\dint}{\operatorname{d}}
\newcommand{\conv}{\operatorname{conv}}

\numberwithin{equation}{section}

\title[A random cell splitting scheme on the sphere]{A random cell splitting scheme on the sphere}

\author[C.\ Deuß]{Christian Deuß}
\address{Christian Deuß: Faculty of Mathematics, NA 3/28, Ruhr University Bochum, Germany}
\email{christian.deuss@rub.de}

\author[J.\ H\"orrmann]{Julia H\"orrmann}
\address{Julia H\"orrmann: Faculty of Mathematics, NA 3/69, Ruhr University Bochum, Germany}
\email{julia.hoerrmann@rub.de}

\author[C.\ Th\"ale]{Christoph Thäle}
\address{Christoph Th\"ale: Faculty of Mathematics, NA 3/68, Ruhr University Bochum, Germany}
\email{christoph.thaele@rub.de}

\date{}

\begin{document}

\begin{abstract}
A random recursive cell splitting scheme of the $2$-dimensional unit sphere is considered, which is the spherical analogue of the STIT tessellation process from Euclidean stochastic geometry. First-order moments are computed for a large array of combinatorial and metric parameters of the induced splitting tessellations by means of martingale methods combined with tools from spherical integral geometry. The findings are compared with those in the Euclidean case, making thereby transparent the influence of the curvature of the underlying space. Moreover, the capacity functional is computed and the point process that arises from the intersection of a splitting tessellation with a fixed great circle is characterized. 
\end{abstract}

\maketitle

\section{Introduction}

Imagine an apple that shall be cut into pieces. We do this in the following way. In a first step, we perform a straight cut through the centre, which splits the apple into two pieces. Next, we select one of them and cut it again by a straight cut through the apple's centre into two pieces. Now, we select one of these three pieces, cut it in the same way and continue the procedure. After some time has passed, we stop the cutting process, put together all the pieces and observe a tessellation on the surface of the apple that is induced by the cuts, see Figure \ref{fig:Apple} for an illustration. 

\medbreak

\begin{figure}[t]
\begin{center}
\includegraphics[width=0.3\columnwidth]{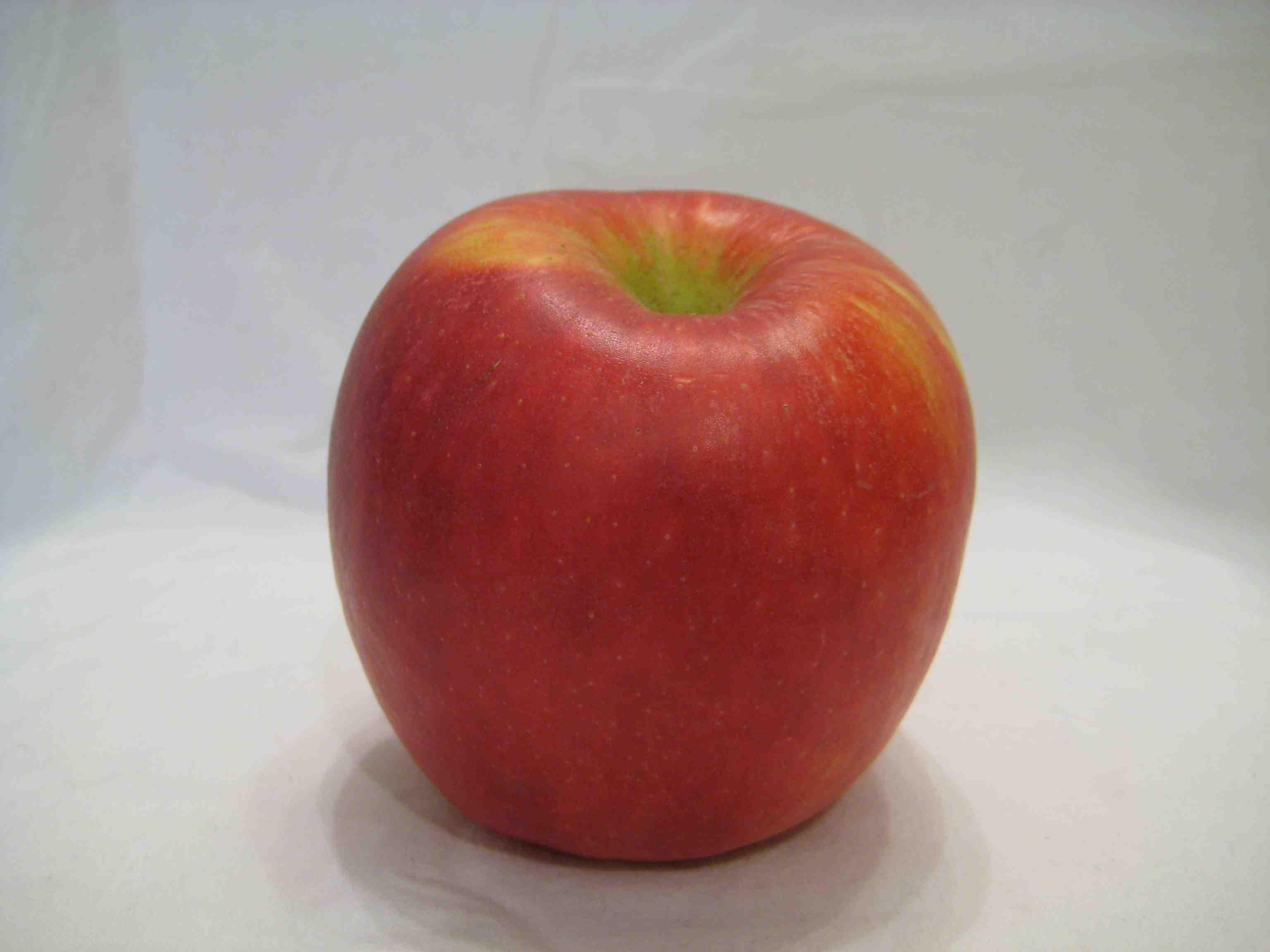}\quad
\includegraphics[width=0.3\columnwidth]{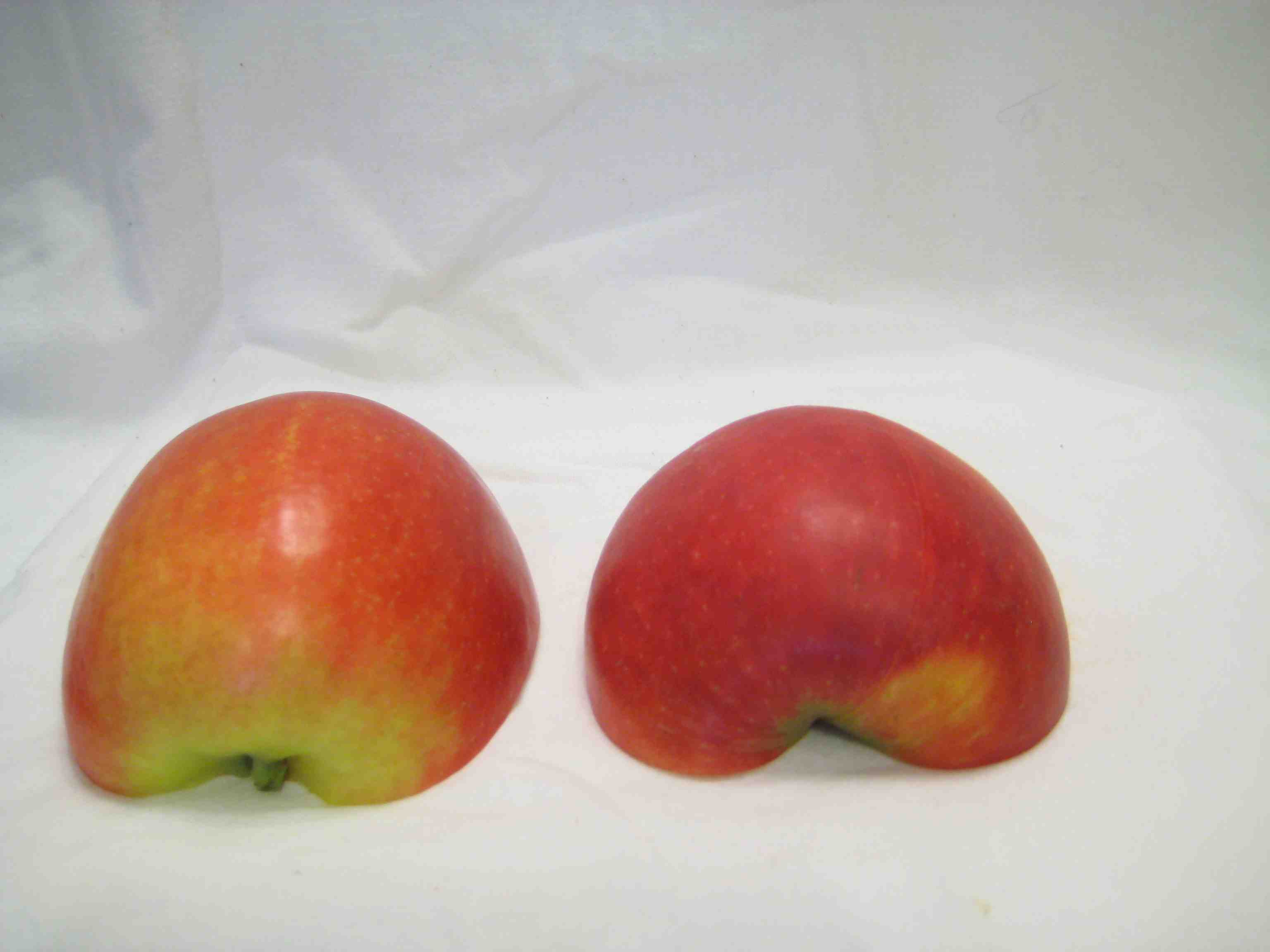}\quad
\includegraphics[width=0.3\columnwidth]{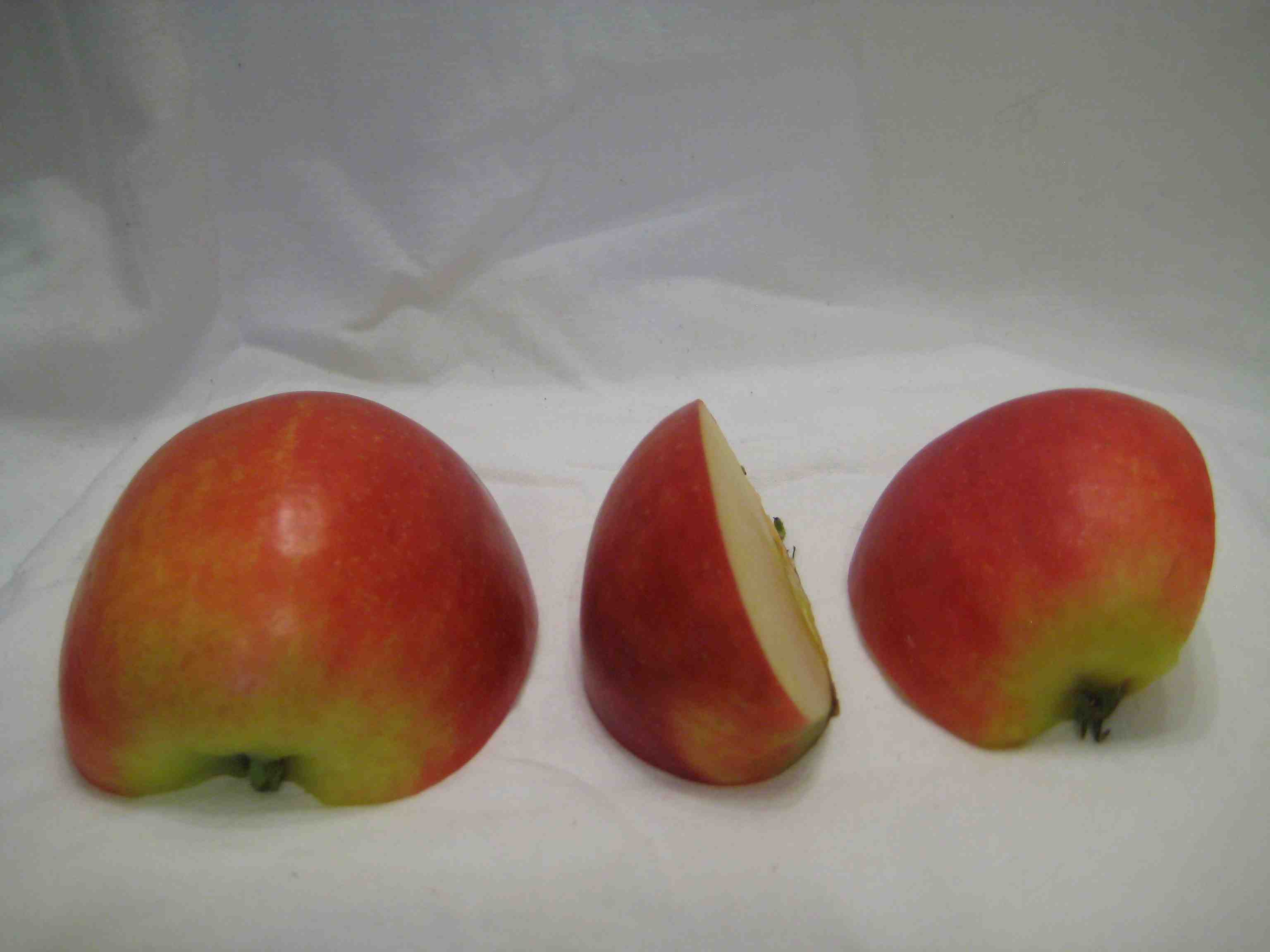}\\
\vspace{0.38cm}
\includegraphics[width=0.3\columnwidth]{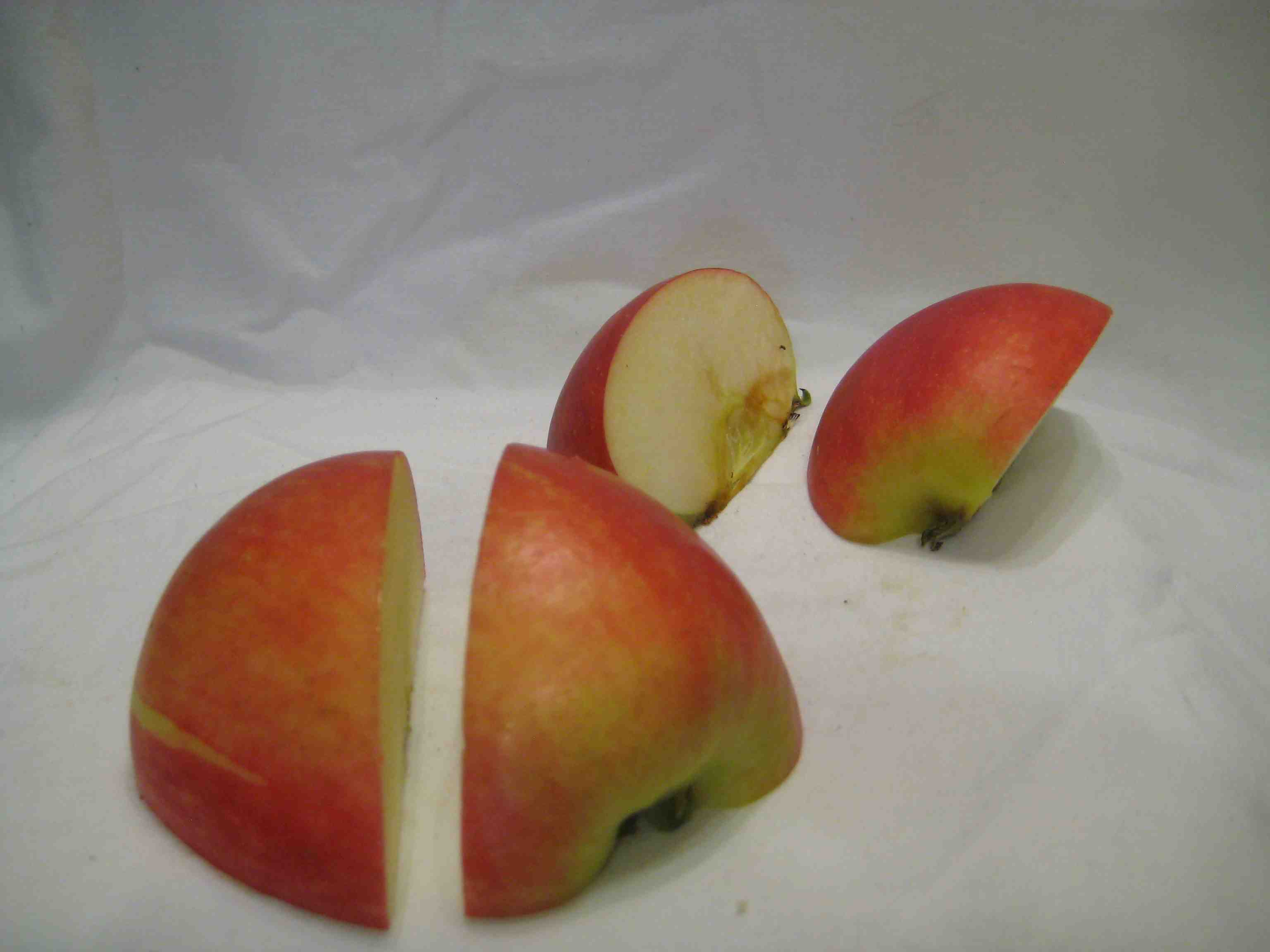}\quad
\includegraphics[width=0.3\columnwidth]{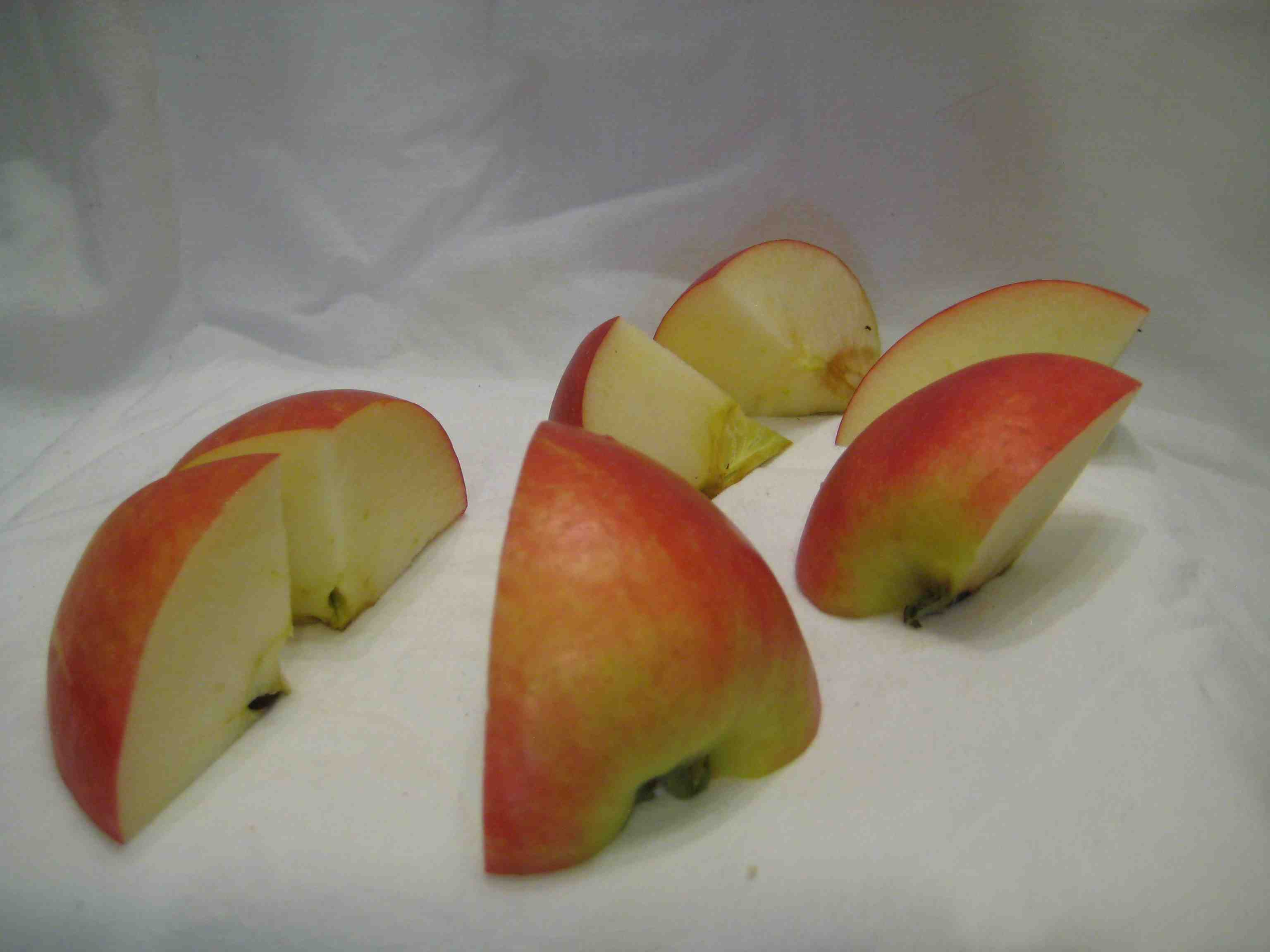}\quad
\includegraphics[width=0.3\columnwidth]{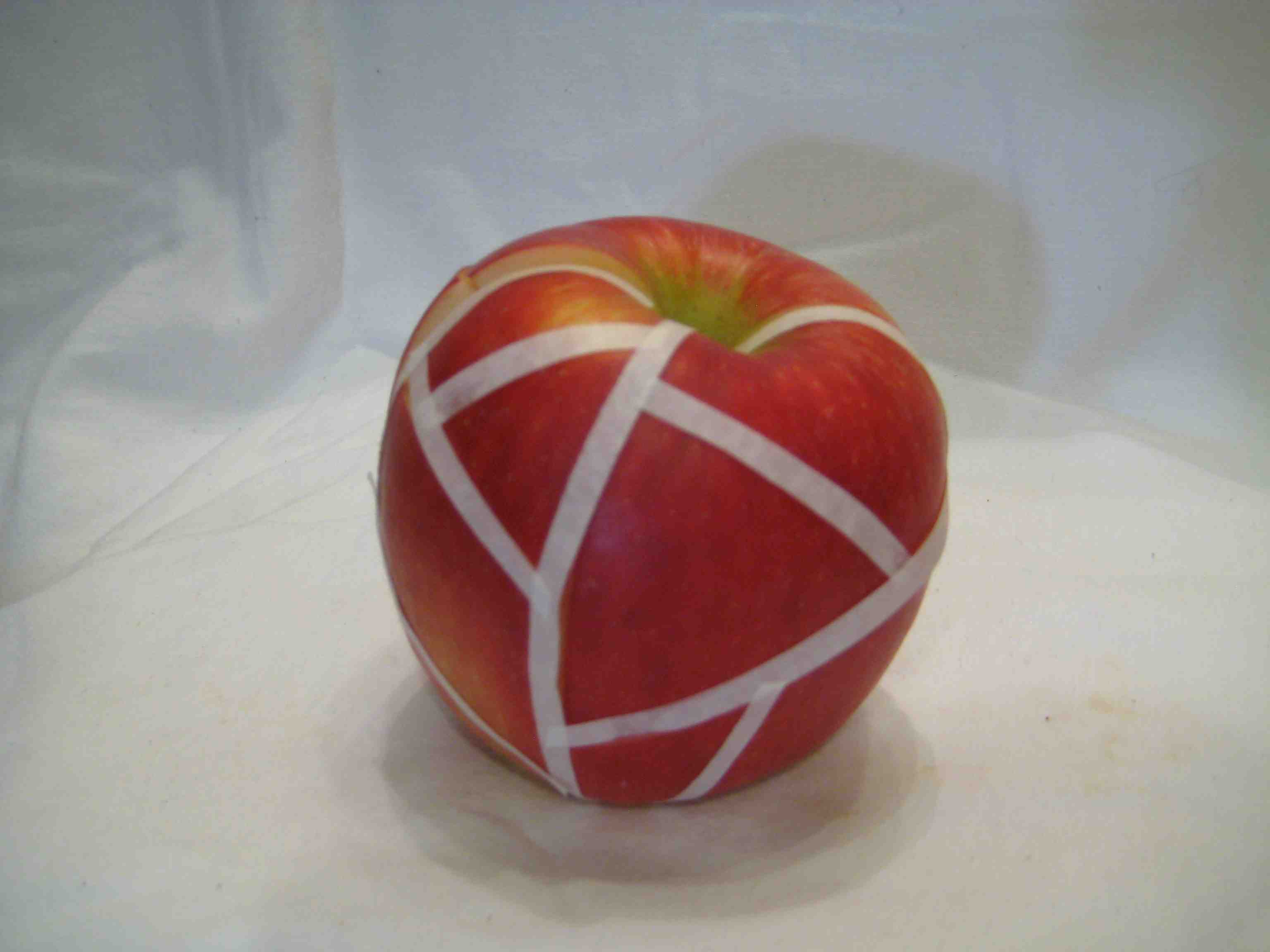}
\end{center}
\caption{Cutting of an apple.}
\label{fig:Apple}
\end{figure}

To model the situation just described, we identify the apple with the unit ball in $\RR^3$ and its surface with the $2$-dimensional unit sphere $\Sp^2$. The time we wait until an existing piece $c$ is cut is assumed to be random and exponentially distributed with parameter $\sigma_2([c])$. Here, $\sigma_2$ denotes the normalized spherical Lebesgue measure and $[c]:=\{u\in\Sp^2:c\cap u^\perp\neq\emptyset\}$. Finally, we assume that also the unit normal vector of the plane used to cut $c$ is random and has distribution $\sigma_2(\,\cdot\,\cap[c])/\sigma_2([c])$. The main purpose of this article is to study the probabilistic and geometric properties of the random tessellation process on $\Sp^2$ that is induced by the described cell splitting scheme, whose formal continuous time Markovian definition will be presented in Section \ref{subsec:Splitting} below (but see Figure \ref{fig:ST} and Figure \ref{fig:Process} for illustrations).

One motivation to study this model comes from Euclidean stochastic geometry, where instead of $\Sp^2$ a prescribed convex body $W\subset\RR^2$ in the plane is subdivided by random line segments. The resulting tessellation of $W$ is then extended to a stationary random tessellation in the whole plane by consistency and is known as a so-called STIT tessellation. After its introduction by Nagel and Wei\ss\ \cite{NagelWeiss2005} this model has attracted considerable interest because of the numerous analytically available results, see \cite{LachiezeRey,MeckeNagelWeiss2011,NagelWeiss2007,SchreiberThaelePMS,SchreiberThaeleBernoulli,SchreiberThaeleAOP,TWN,WeissNagelOhser}. It is our aim to study the spherical counterpart of planar STIT tessellations and the effects induced by the curvature of the underlying state space. While STIT tessellations in the plane are \underline{st}able under the operation of rescaled \underline{it}eration of tessellations, this is not the case for their spherical counterparts, since no such rescaling is possible on $\Sp^2$. For this reason, we decided to call them \textit{splitting tessellations}, as they arise as a result of the described splitting scheme.

While random tessellations of Euclidean spaces, especially of the Euclidean plane, are a classical topic in stochastic geometry (see Chapter 10 in \cite{SchneiderWeil} and the references given therein), random tessellations on the sphere have not found equal attention in the existing literature. However, random tessellations of $\Sp^2$ or higher-dimensional spherical spaces by great circles were studied in \cite{ArbeiterZaehle,MilesSphere,Santalo} and also the recent work \cite{HugSchneiderConical} on so-called conical random tessellation is closely related to this model. Spherical random Voronoi tessellations and their applications were investigated in \cite{MilesSphere,SugiharaVoronoiSphere,VoronoiSphereTanemura,VoronoiSphereAppl}. On the other hand, the splitting tessellations we study fall into a different class of models that have not been considered so far. They are by construction and in contrast to great circle or Voronoi tessellations not side-to-side, meaning that the intersection of two tessellation cells is not necessarily a common side of both cells. This gives rise to a couple of new geometric effects that are exploited within the present paper as well and were another source of motivation for us to study splitting tessellations on $\Sp^2$. We also take the opportunity to mention that the splitting tessellations we introduce and study can be regarded as a first attempt to stochastic-geometric models for crack structures on surfaces.

\begin{figure}[t]
\begin{center}
\includegraphics[trim = 45mm 167mm 50mm 15mm, clip, width=0.5\columnwidth]{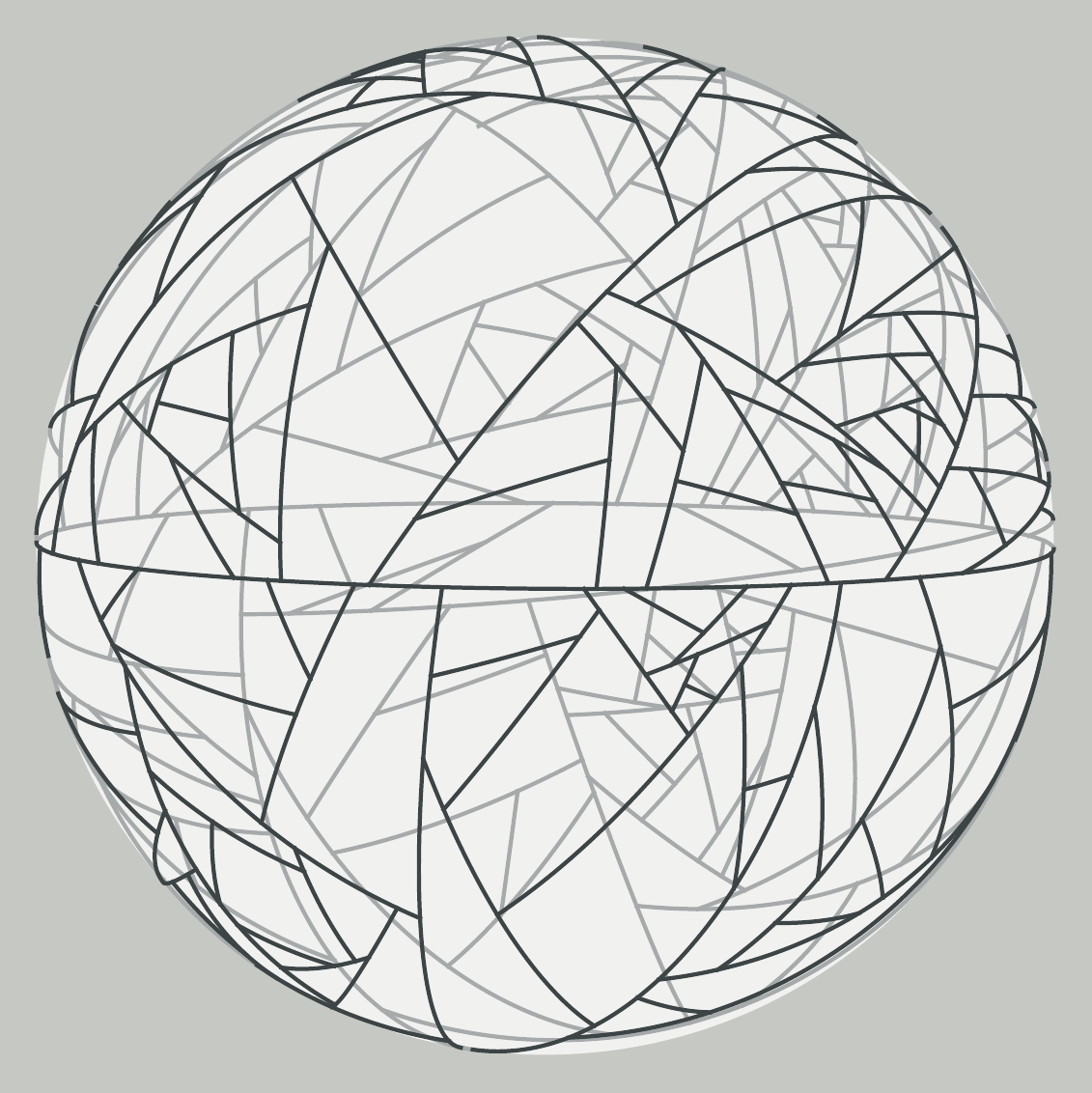}
\end{center}
\caption{Illustration of a splitting tessellation.}
\label{fig:ST}
\end{figure}

\medbreak

Let us briefly describe the findings of our paper. We first compute three total length parameters, namely the mean total edge length, the mean total side length of the cells as well as the mean length of all so-called maximal segments of a splitting tessellation. It turns out that, in contrast to the Euclidean case mentioned above, these three quantities do not coincide. We also compute the intensities of a number of point processes that are associated to geometric objects of splitting tessellations. We then determine a large array of mean adjacencies, including quantities like the mean number of vertices on the boundary of a `typical' cell or the mean number of vertices that are located on the `typical' cell-side. In the large time asymptotics, all these mean values behave like their Euclidean counterparts, while for comparably `small' time parameters the curvature of the sphere shows a significant influence. Beside the first-order behaviour of splitting tessellations, we also analyse some of their more sophisticated properties. We determine the capacity functional, which is the suitable substitute for random sets of the distribution function of a real-valued random variable. As an application we show that the point process that arises if a splitting tessellation is intersected by a fixed great circle is the concatenation of two independent Poisson point processes plus a pair of deterministic antipodal points. On the technical side, the proofs of our results make an extensive use of the martingale property of a family of stochastic processes that can be associated to the splitting tessellation construction. This is then combined with tools from spherical integral geometry as well as with combinatorial arguments. We restrict our attention to the $2$-dimensional unit sphere, although some of our results continue to hold in higher-dimensional spherical spaces as well. However, in dimension $2$ the most explicit results are available and we doubt, for example, that the adjacencies presented in Theorem \ref{thn:Adjacencies} below can be computed for general dimensions $d\geq2$.

\medbreak

Our paper is structured as follows. In Section \ref{subsec:SphericalGeom} we provide some background material related to spherical integral geometry. Our polygonal cell splitting scheme and the notion of a splitting tessellation are formally introduced in Section \ref{subsec:Splitting}, while in Section \ref{subsec:Martingales} we construct a family of associated martingales. Different set classes as well as the concept of Palm distributions are reviewd in Section \ref{subsec:SetClasses}. The main results of this paper are presented in Section \ref{sec:Results} and the final Section \ref{sec:Proofs} contains their proofs.

\section{Preliminaries}\label{sec:Preliminaries}

\subsection{Spherical integral geometry}\label{subsec:SphericalGeom}

By $\Sp^2$ we denote the Euclidean unit sphere, which inherits a Riemannian structure from the ambient space $\RR^3$. The geodesic distance on $\Sp^2$ is denoted by $d_s(\,\cdot\,,\,\cdot\,)$. The collection of great circles on $\Sp^2$, that is, intersections of $\Sp^2$ with planes in $\RR^3$ through the origin, is $\sG(\Sp^2)$. This space carries an invariant probability measure $\tau$ given by the relation
$$
\int_{\sG(\Sp^2)}F(g)\,\tau(\dint\! g) = \int_{\Sp^2}F(\Sp^2\cap u^\perp)\,\sigma_2(\dint\! u)\,,
$$
where $F:\Sp^2\to\RR$ is a non-negative measurable function and $\sigma_2$ stands for the normalized spherical Lebesgue measure, i.e., $\sigma_2(\Sp^2)=1$. Moreover, by $\sigma_1$ we denote the spherical length measure which is normalized in such a way that $\sigma_1(g)=1$ for all great circles $g\in\sG(\Sp^2)$.

Each great circle $g\in\sG(\Sp^2)$ separates $\Sp^2$ into two hemispheres. By a spherically convex polygon we understand the intersection of finitely many closed hemispheres and we denote by $\sP(\Sp^2)$ the collection of all such polygons. For technical reasons, we will also include the whole sphere $\Sp^2$ as a degenerate case in $\sP(\Sp^2)$. We supply the space $\sP(\Sp^2)$ with the Borel $\sigma$-field $\cB(\sP(\Sp^2))$ induced by the spherical Hausdorff distance. The interior and the boundary of a spherically convex polygon $p\in\sP(\Sp^2)$ are denoted by ${\rm int}(p)$ and ${\rm bd}(p)$, respectively, in what follows. If $p\in\sP(\Sp^2)$ is a spherically convex polygon with ${\rm int}(p)\neq\emptyset$ and if $p$ has representation $p=S_1\cap\ldots\cap S_n$ with closed hemispheres $S_1,\ldots,S_n$ that are bounded by great circles $g_1,\ldots,g_n\in\sG(\Sp^2)$, we call $p\cap g_m$ a side and $p\cap g_m\cap g_\ell$ a corner of $p$, $m,\ell\in\{1,\ldots,n\}$, provided that the intersections are non-empty.

Next, we recall from Chapter 6.5 in \cite{SchneiderWeil} the definition of the spherical intrinsic volumes $v_0,v_1$ and $v_2$ on $\sP(\Sp^2)$. They are defined as coefficients of the spherical Steiner formula. It says that for $p\in\sP(\Sp^2)$ with $p\neq\Sp^2$ and $\varepsilon\in(0,\pi/2)$,
$$
\sigma_2(\{x\in\Sp^2:d_s(x,p)\le \varepsilon\})=(1-\cos (\varepsilon))\,v_0(p)+\sin (\varepsilon)\, v_1(p)+v_2(p)\,,
$$
where $d_s(x,p)=\min\{d_s(x,y):y\in p\}$, see \cite[Theorem 6.5.1]{SchneiderWeil}. In particular, $v_2(p)=\sigma_2(p)$ is the normalized spherical Lebesgue measure of $p$ and we also have the relations $\tau([p]) = \sigma_1(\partial p)$ and $\sigma_1(s)=v_1(s)$ for a spherically convex polygon $p$ with ${\rm int}(p)\neq\emptyset$ that contains no great circle and a spherical line segment $s$. For convenience, we also put $v_2(\Sp^2):=1$ and $v_1(\Sp^2)=v_0(\Sp^2):=0$. The so defined functionals are invariant under rotations (and reflections) of $\Sp^2$, they are continuous with respect to the spherical Hausdorff distance and they are additive in the sense that $v_i(p\cup q)-v_i(p\cap q)=v_i(p)+v_i(q)$ for all $p,q\in\sP(\Sp^2)$ for which $p\cup q\in\sP(\Sp^2)$ and $i\in\{0,1,2\}$. In fact, the spherical intrinsic volumes $v_0,v_1,v_2$ form a basis of the vector space of all such functionals on $\sP(\Sp^2)$, cf.\ \cite[Theorem 6.5.4]{SchneiderWeil} and \cite[Theorem 11.3.1]{KlainRota}. This Hadwiger-type theorem can be used to give a short proof of the following result, which is the spherical counterpart of the well-known Crofton formula.

\begin{lemma}[Spherical Crofton formula; \cite{SchneiderWeil}, page 261]\label{lem:Crofton}
Let $p\in\sP(\Sp^2)$. Then,
\begin{align*}
\int_{\sG(\Sp^2)}v_0(p\cap g)\,\tau(\dint\! g)=v_1(p)\qquad\text{and}\qquad\int_{\sG(\Sp^2)}v_1(p\cap g)\,\tau(\dint\! g)=v_2(p)\,.
\end{align*}
\end{lemma}
\begin{proof}
Define a functional $\varphi$ on $\sP(\Sp^2)$ by 
$$
\varphi(p):=\int_{\sG(\Sp^2)}v_0(p\cap g)\,\tau(\dint\! g)
$$
and notice that $\varphi$ is invariant under rotations, continuous with respect to the spherical Hausdorff distance and additive in the sense explained above. Thus, by Theorem 6.5.4 in \cite{SchneiderWeil} or Theorem 11.3.1 in \cite{KlainRota}, there are constants $c_0,c_1,c_2\in\RR$ such that
$$
\varphi(p)=c_0v_0(p)+c_1v_1(p)+c_2v_2(p)
$$
for all $p\in\sP(\Sp^2)$. To determine the values of these constants, we first plug in for $p$ a single point $x\in\Sp^2$. Since $\varphi(\{x\})=0$ and $v_1(\{x\})=v_2(\{x\})=0$ we have that $c_0=0$. Next, we plug in for $p$ a great circle $h\in\sG(\Sp^2)$. By \cite[Equation (6.5.1)]{SchneiderWeil} one has that $\varphi(h)=1$ and $v_2(h)=0$ and thus $c_1=1$. Finally, we choose $p$ as the sphere $\Sp^2$ and notice that $\varphi(\Sp^2)=0$, since $v_0(\Sp^2\cap g)=v_0(g)=0$ for any $g\in\sG(\Sp^2)$, see \cite[Equation (6.53)]{SchneiderWeil}. This implies $c_2=0$ and completes the proof of the first formula. The second one follows by a similar argument and for this reason we skip the details.
\end{proof}

For further results and background material on spherical integral geometry (also in higher-dimensional spherical spaces) we refer to Chapter 11 in \cite{KlainRota}, to Chapter 6.5 in \cite{SchneiderWeil} and also to the monograph \cite{Santalo}.

\subsection{Tessellations and splitting tessellations}\label{subsec:Splitting}

Recall that $\sP(\Sp^2)$ stands for the space of all spherically convex polygons.

\begin{definition}
A tessellation $T$ of $\Sp^2$ is a finite subset of $\sP(\Sp^2)$ with the following properties:
\begin{itemize}
\item[(i)] ${\rm int}(p)\cap{\rm int}(p')=\emptyset$ for all $p,p'\in T$ with $p\neq p'$,
\item[(ii)] $\bigcup_{p\in T}p=\Sp^2$.
\end{itemize}
The elements of $T$ are called the cells of the tessellation $T$.
\end{definition}

In our paper we will identify a tessellation $T$ with the closed set
$$
\bigcup_{p\in T}{\rm bd}(p)\subset\Sp^2
$$
and write $\bZ(T)$ for the set of cells of $T$. Then the space $\sT(\Sp^2)$ of tessellations of $\Sp^2$ can be regarded as a subspace of the space of closed subsets of $\Sp^2$ and in this way, $\sT(\Sp^2)$ can be equipped with a Borel $\sigma$-field $\cB(\sT(\Sp^2))$ that is induced by the trace of the usual Fell topology on the space of closed subsets of $\Sp^2$, see \cite{SchneiderWeil}. A spherical random tessellation can now be defined as a measurable mapping from some probability space $(\Omega,\cA,\PP)$ taking values in the measurable space $(\sT(\Sp^2),\cB(\sT(\Sp^2)))$.

For a Borel set $B\in\cB(\Sp^2)$ we introduce the notation $[B]:=\{g\in\sG(\Sp^2):B\cap g\neq \emptyset\}$. Now, we define for a tessellation $T\in\sT(\Sp^2)$, a cell $p\in\bZ(T)$ and a great circle $g\in[p]$ a new tessellation by
\begin{align*}
\oslash_{p,g}(T):=\bigcup_{p'\in((\bZ(T)\setminus p)\cup\{p\cap g^+,p\cap g^-\})}{\rm bd}(p')\,,
\end{align*}
where $g^{\pm}$ are the two closed hemispheres into which $\Sp^2$ is separated by $g$. In other words, $\oslash_{p,g}(T)$ is the tessellation that arises from $T$ if the cell $p$ is split by the great circle $g$.

To proceed, let us dissect $\Sp^2$ into the two hemispheres $\Sp^2_\pm$ determined by the equator, which is denoted by $A$ in what follows, i.e., $A=\Sp^2\cap\{(x,y,z)\in\RR^3:z=0\}$, $\Sp^2_+:=\Sp^2\cap\{(x,y,z)\in\RR^3:z\geq 0\}$ and $\Sp^2_-:=\Sp^2\cap\{(x,y,z)\in\RR^3:z\leq 0\}$. We are now prepared to formally introduce splitting tessellations and the splitting tessellation process on $\Sp^2$.

\begin{definition}
By the splitting tessellation process $(Y_t)_{t\geq 0}$ with initial tessellation $Y_0:=A$ we understand the continuous time Markov process on $\sT(\Sp^2)$ whose generator is given by
\begin{align}\label{SphercialTessellationGenerator}
\mathbb{L} f(T) := \sum_{p\in\bZ(T)}\int_{[p]} \big[\,f(\oslash_{p,g}(T))-f(T)\,\big]\,\tau(\dint\!g)\,,
\end{align}
where $T\in\sT(\Sp^2)$ and $f:\sT(\Sp^2)\to\RR$ is bounded and measurable. By a splitting tessellation with time parameter $t\geq 0$ we mean the spherical random tessellation $Y_t$.
\end{definition}

\begin{figure}[t]
\includegraphics[trim = 45mm 175mm 50mm 15mm, clip, width=0.3\columnwidth]{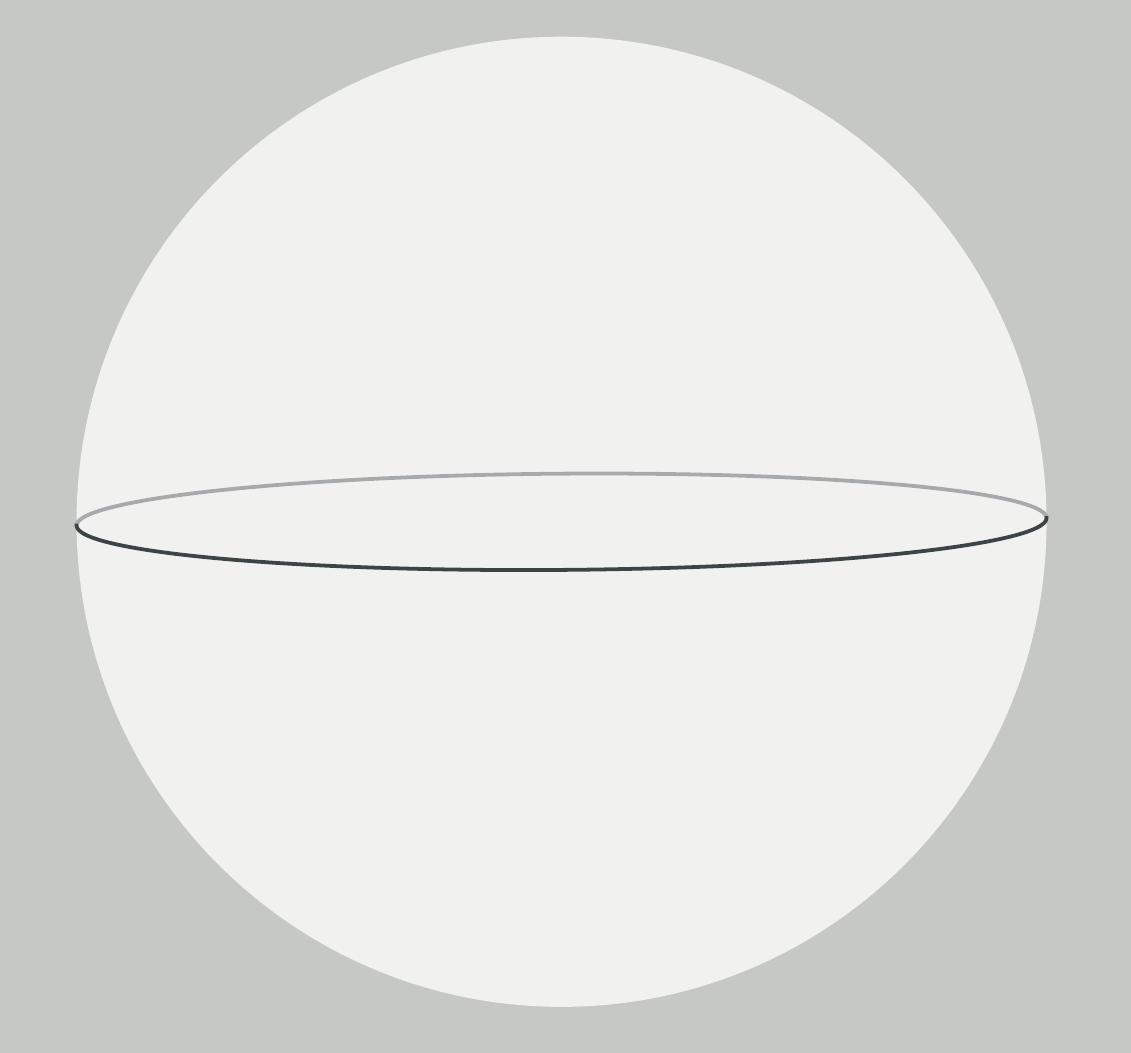}
\includegraphics[trim = 45mm 175mm 50mm 15mm, clip, width=0.3\columnwidth]{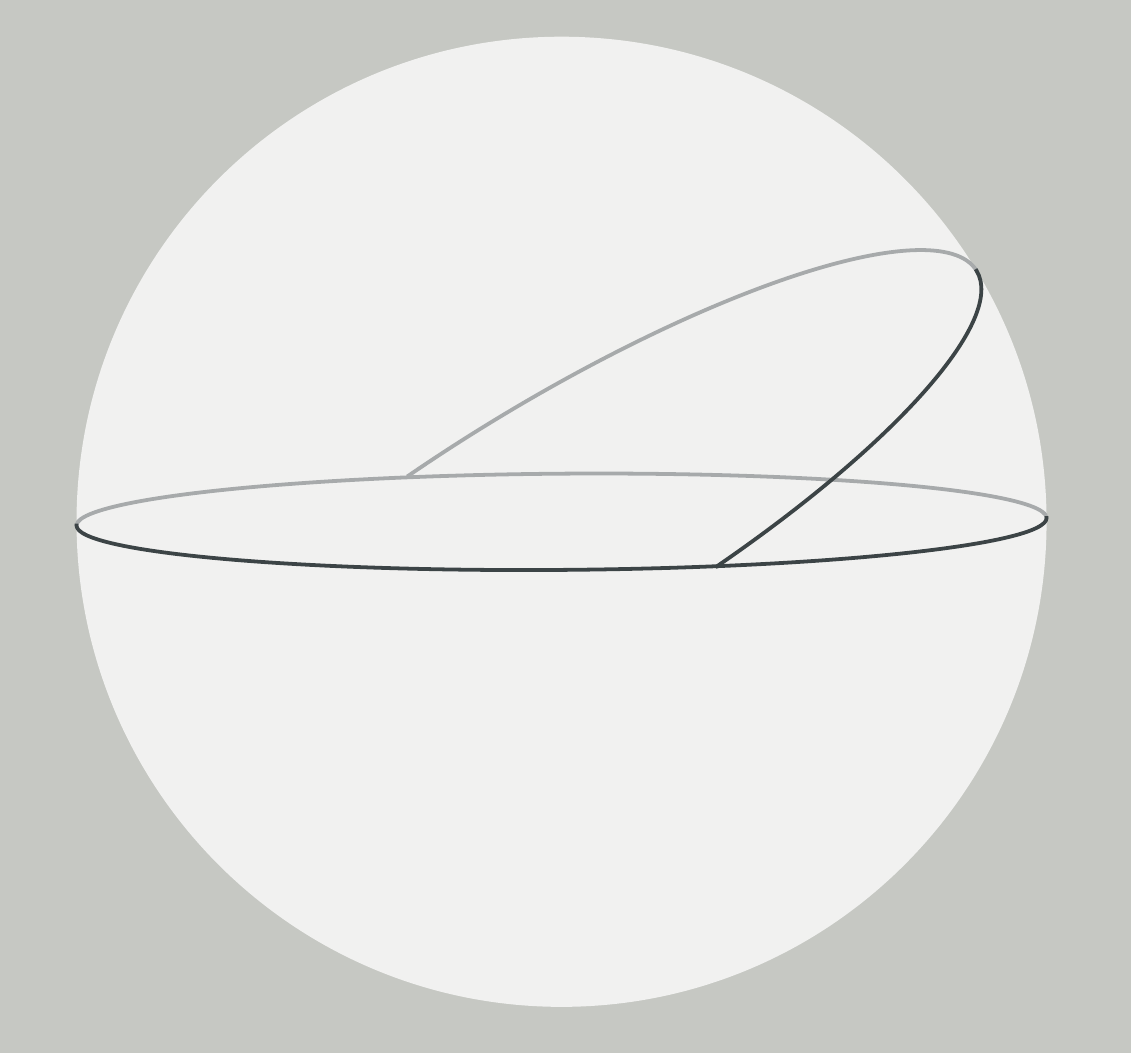}
\includegraphics[trim = 45mm 175mm 50mm 15mm, clip, width=0.3\columnwidth]{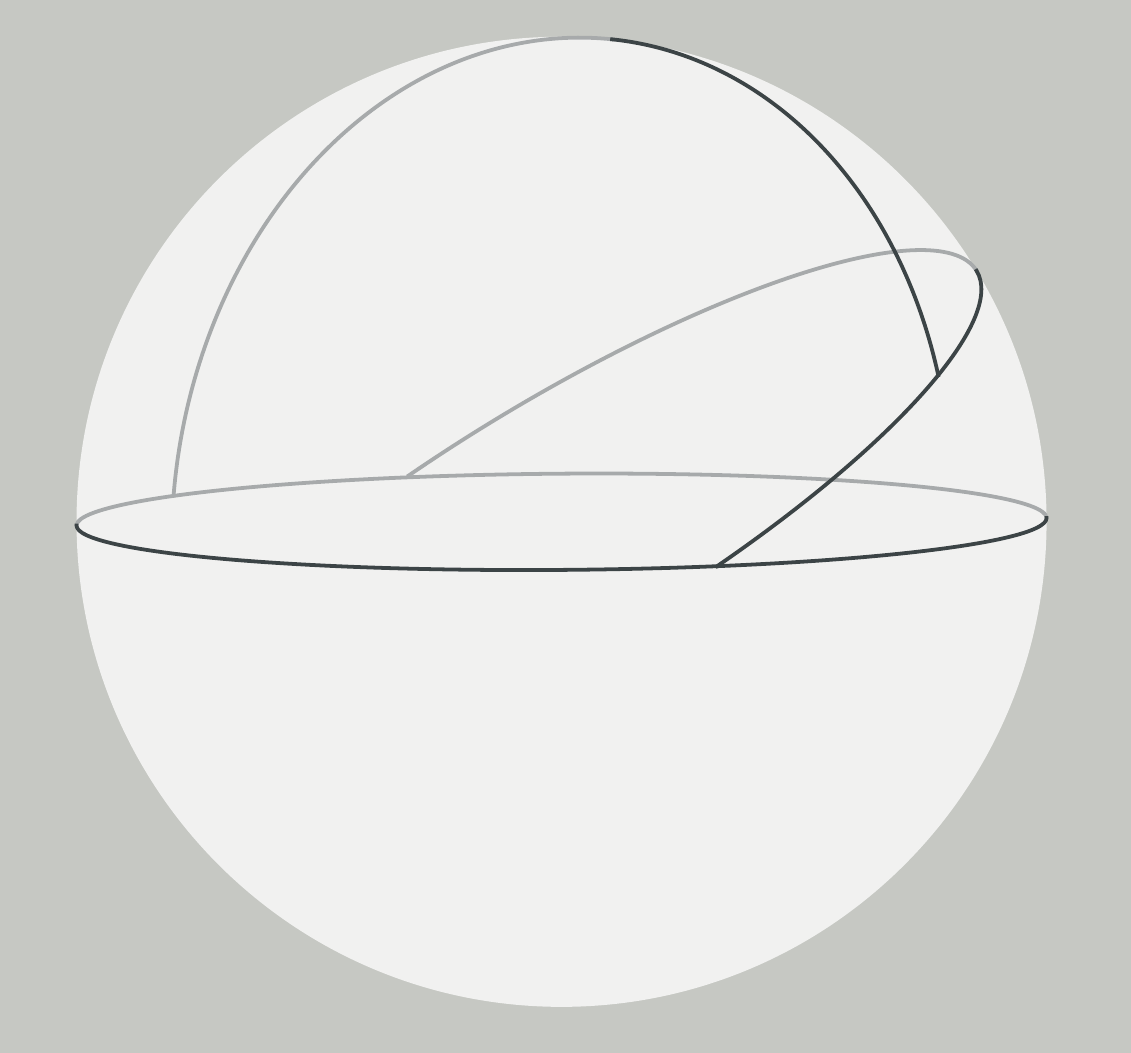}\\
\vspace{0.1cm}
\includegraphics[trim = 45mm 175mm 50mm 15mm, clip, width=0.3\columnwidth]{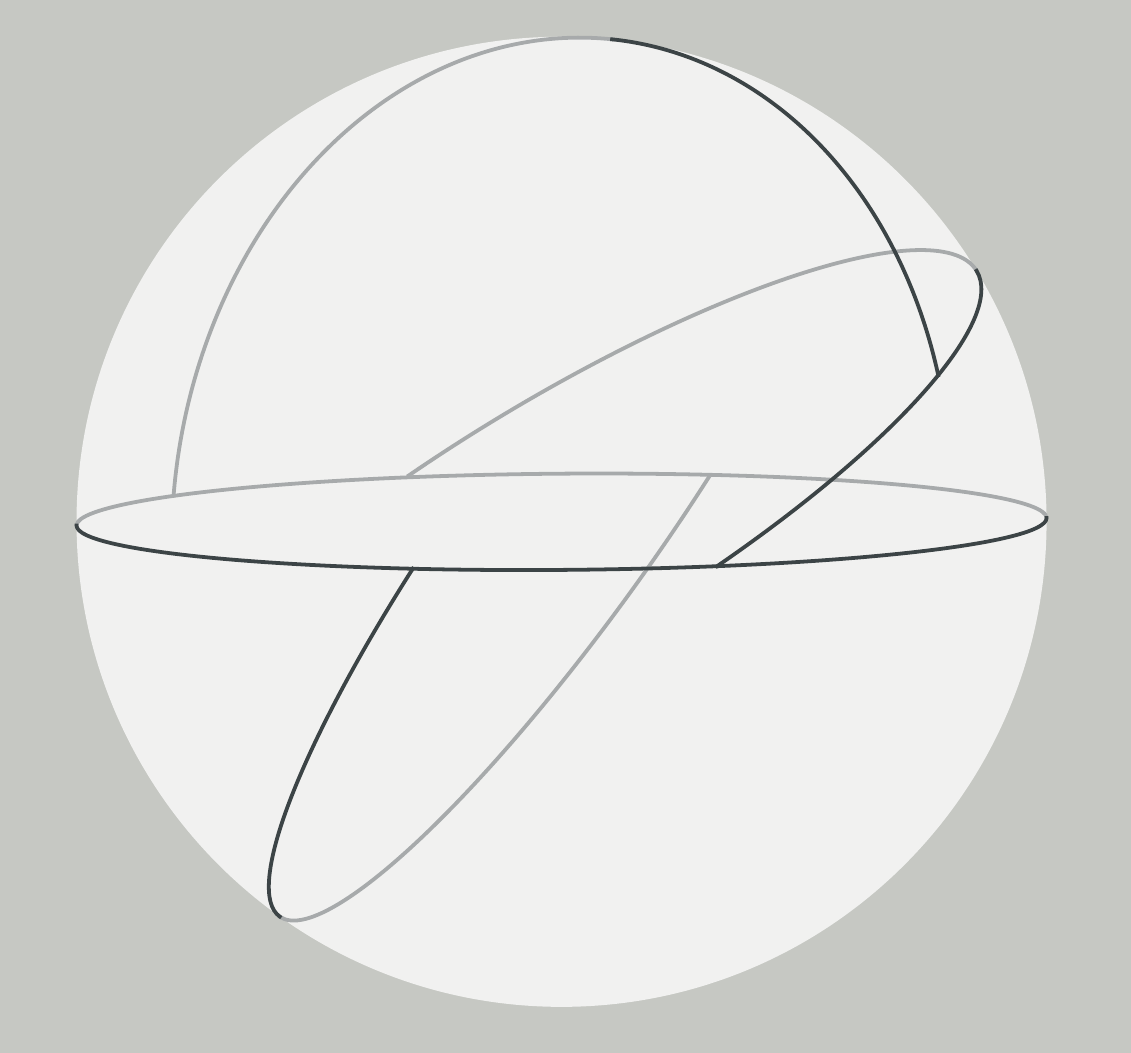}
\includegraphics[trim = 45mm 175mm 50mm 15mm, clip, width=0.3\columnwidth]{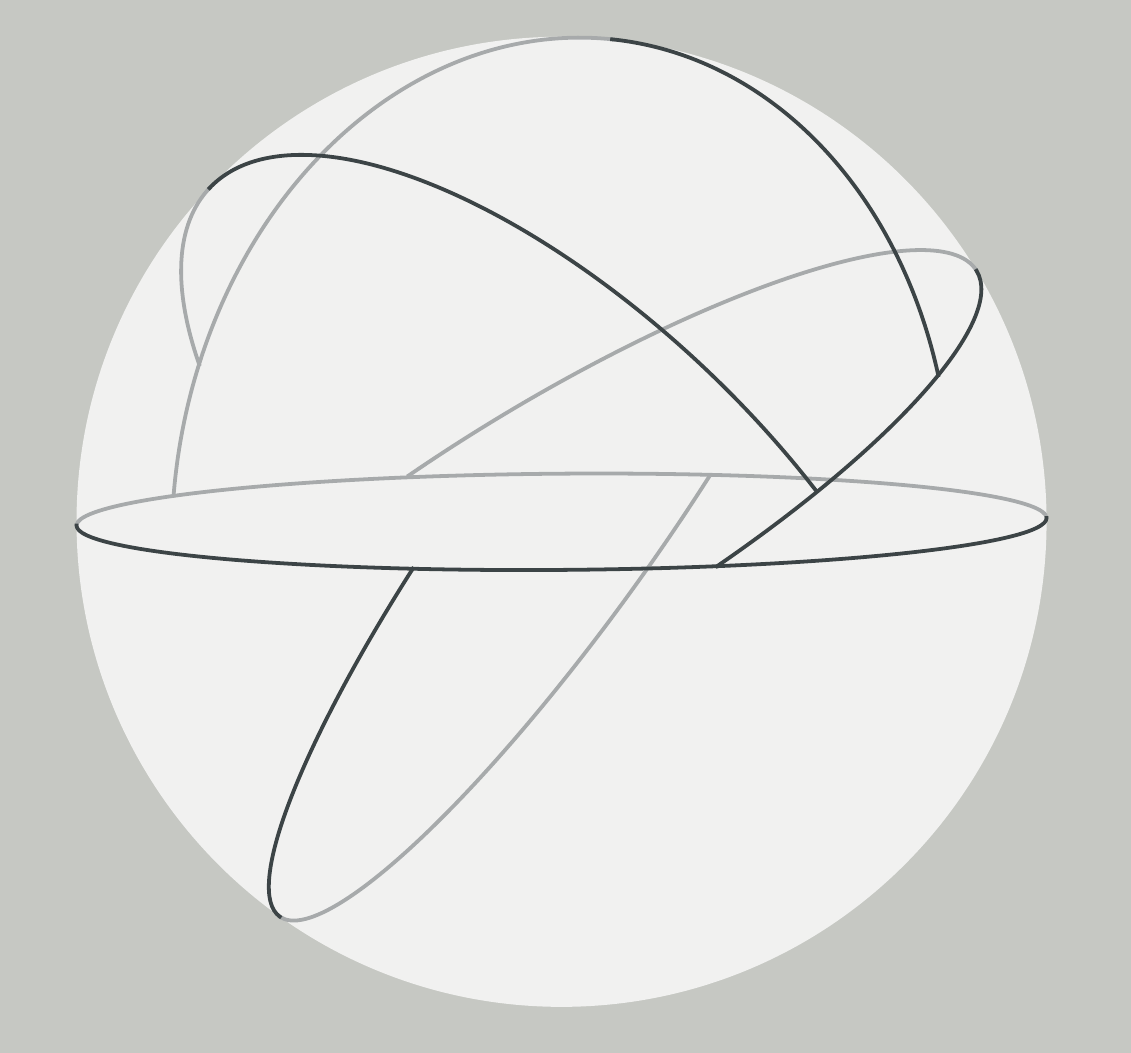}
\includegraphics[trim = 45mm 175mm 50mm 15mm, clip, width=0.3\columnwidth]{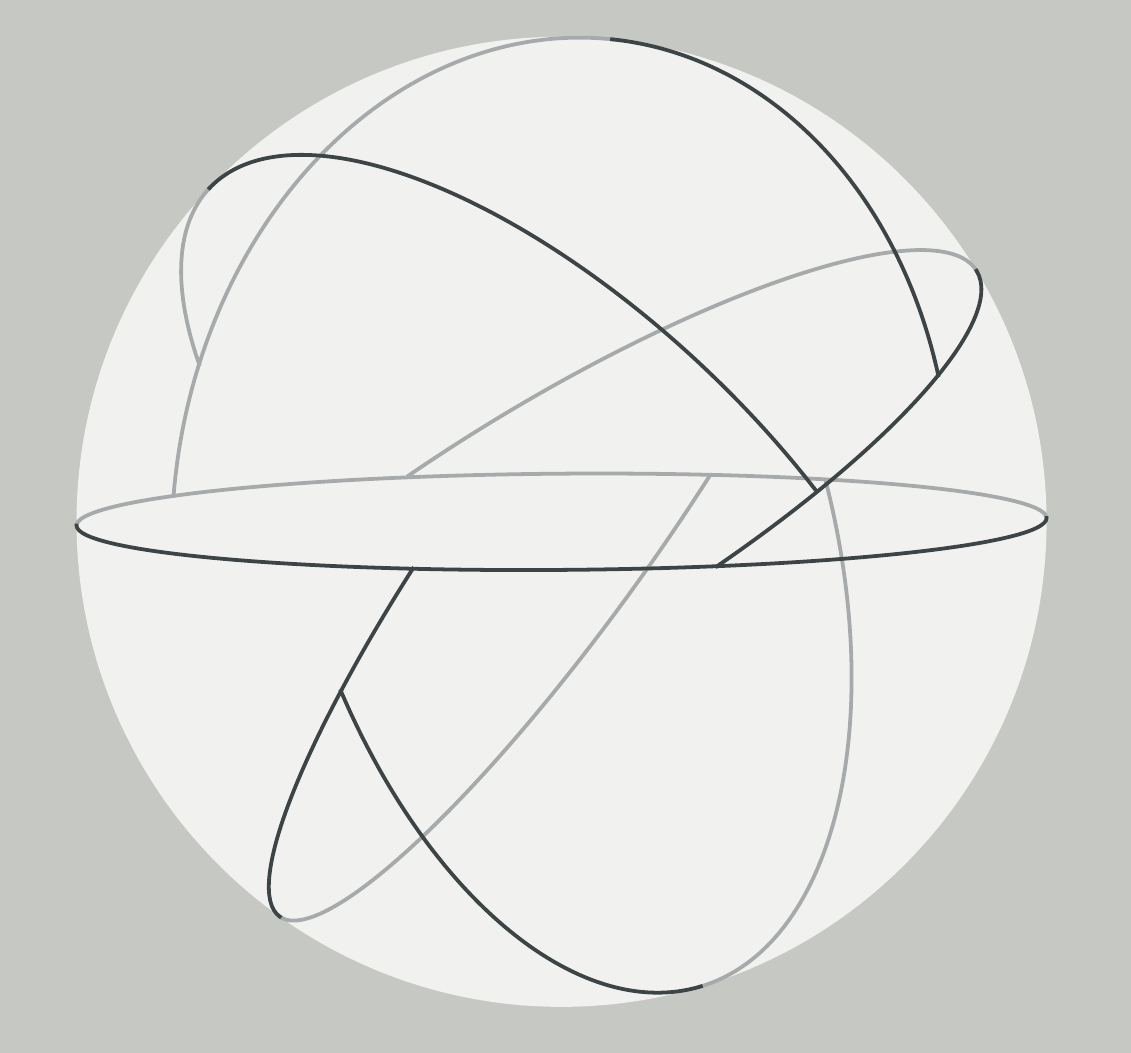}
\caption{Illustration of the splitting tessellation process on $\Sp^2$ at different time instants starting at $t=0$ and with the initial tessellation $Y_0=A$.}
\label{fig:Process}
\end{figure}

In spite of the choice of the initial tessellation $Y_0=A$ that consists of the two cells $\Sp^2_\pm$, the splitting tessellations $Y_t$ defined above are the clear spherical analogues of the STIT tessellations introduced in \cite{NagelWeiss2005} and considered in Euclidean stochastic geometry. Their dynamic can be described as in the introduction as a continuous-time branching process on $\sP(\Sp^2)$. That is, at time $t=0$ the tessellation consist of the equator $A$ and the two cells $\Sp_+^2$ and $\Sp_-^2$. Both have independent and exponentially distributed random lifetimes with parameter $1$. If one of the cells, say $\Sp_+^2$, dies out, a great circle $g\in[\Sp_+^2]=\sG(\Sp^2)$ is selected according to the distribution $\tau$ which splits $\Sp_+^2$ into the two sub-cells $\Sp^2_+\cap g^+$ and $\Sp^2_+\cap g^-$ that have $\Sp^2_+\cap g$ as a common side. Within these two sub-cells the construction starts anew with the lifetimes of $\Sp^2_+\cap g^+$ and $\Sp^2_+\cap g^-$ being independent and exponentially distributed with parameter $\tau([\Sp^2_+\cap g^+])$ and $\tau([\Sp^2_+\cap g^-])$, respectively (Figure \ref{fig:Process} illustrates the splitting process). In particular, we notice that once an existing cell $p$ is split by a great circle $g\in[p]$, the further splitting constructions within the two `daughter' cells $p\cap g^+$ and $p\cap g^-$ are (conditionally) independent and follow the same rules.

\medskip

Let us finally motivate our choice of the initial tessellation $Y_0=A$. Imagine we would have started with the empty initial tessellation that has $\Sp^2$ as its single cell. After some exponential random time has passed, $\Sp^2$ would have been split by a great circle $g\in\sG(\Sp^2)$. By rotation invariance, one can assume that $g=A$ and we notice that $A$ is the only great circle that is constructed during the splitting process. In other words this means that starting the splitting process with the initial tessellation $Y_0=A$ induces no loss of generality.

\subsection{A family of associated martingales}\label{subsec:Martingales}

In what follows we consider functionals on $\sT(\Sp^2)$ of the form
\begin{align*}
\Sigma_{\phi}(T):=\sum\limits_{p\in\bZ(T)} \phi (p)\,, \qquad T\in\sT(\Sp^2)\,,
\end{align*}
where $\phi:\sP(\Sp^2)\to\RR$ is bounded and measurable. The standard theory of Markov processes can now be used to construct a class of martingales that is associated with the splitting tessellation process $(Y_t)_{t\geq 0}$.

\begin{lemma}\label{lem:martingale}
Let $\phi$ be a bounded measurable function on $\sP(\Sp^2)$. Then the random process $(\Sigma_\phi(Y_t))_{t\geq 0}$ given by
\begin{align}\label{eq:Martingale}
\Sigma_{\phi}(Y_t)-\Sigma_{\phi}(Y_0)-\int_{0}^t\mathbb{L}\Sigma_\phi(Y_s)\,\mathrm{d}s\,,\qquad t\ge 0\,,
\end{align}
is a martingale with respect to the filtration induced by $(Y_t)_{t\geq 0}$.
\end{lemma}
\begin{proof}
The proof is the same as that of Proposition 2 in \cite{SchreiberThaeleBernoulli}.
\end{proof}

The martingale property of the random process $(\Sigma_\phi(Y_t))_{t\geq 0}$ will be one of the crucial tools used in our proofs.

\subsection{Set classes, Palm distributions and typical objects}\label{subsec:SetClasses}

We introduce different classes of sets of geometrical objects that are associated with a splitting tessellation. For a splitting tessellation $Y_t$ we define the following classes of primitive objects:
\begin{itemize}
\item[-] $\bZ(Y_t)$ the class of cells of $Y_t$,
\item[-] $\bE(Y_t)$ the class of edges of $Y_t$,
\item[-] $\bV(Y_t)$ the class of vertices of $Y_t$,
\end{itemize}
where by a vertex we mean the non-empty intersection of three different cells of $Y_t$ and an edge of $Y_t$ is a spherical line segment in the boundary of some cell of $Y_t$, which is bounded by two different vertices and has no vertex in its relative interior. We further introduce the notation $\partial\bZ(Y_t)$ to denote the class of cell boundaries of cells of $Y_t$. We also let
\begin{itemize}
\item[-] $\bS(Y_t)$ be the class of cell sides (i.e., $1$-dimensional faces) of $Y_t$ and
\item[-] $\bM(Y_t)$ be the class of maximal segments of $Y_t$.
\end{itemize}
Here, by a maximal segment we understand the maximal union of connected edges of $Y_t$ that are located on a common great circle and have different endpoints. In other words, the set of maximal segments of $Y_t$ is precisely the set of spherical line segments that are constructed until time $t$ (excluding thereby the equator $A$). We regard $\bS(Y_t)$ as a multi-set, meaning that for two cells $p_1,p_2\in\bZ(Y_t)$ that share a common side $s=p_1\cap p_2$, the side $s$ is contained twice in $\bS(Y_t)$. Note that $\bZ(Y_0)=\{\Sp^2_+,\Sp^2_-\}$ and $\bS(Y_0)=\{A,A\}$, while $\bE(Y_0)=\bV(Y_0)=\bM(Y_0)=\emptyset$.

From now on we shall write $\bZ,\partial\bZ,\bE,\bV,\bS$ and $\bM$ instead of $\bZ(Y_t),\partial\bZ(Y_t),\bE(Y_t),\bV(Y_t),\bS(Y_t)$ and $\bM(Y_t)$, respectively, and consider the time parameter $t\geq 0$ as being fixed. 

Let $\bX\in\{\bZ,\bE,\bV,\bS,\bM\}$ be one of the set classes introduced above and for $x\in\bX$ let $m(x)\in\Sp^2$ be a centre function such that $m(\vartheta x)=\vartheta m(x)$ for all rotations $\vartheta$ of $\Sp^2$. If $\bX=\bZ$ then $m(p)$ could be the centre of the smallest spherical cap containing $p\in\bZ$ or if $\bX=\bE$ we can choose $m(e)$ as the midpoint of $e\in\bE$, for example. This gives rise to the rotation invariant point processes $\xi_\bX:=\{m(x):x\in\bX\}$ on $\Sp^2$, whose intensities are denoted by $\lambda_\bX:=\EE|\xi_\bX|$, where $|\,\cdot\,|$ stands for the cardinality of the argument set (in the case $\bX=\bS$ we take into account multiplicities). We have that $\lambda_\bX<\infty$ for all time parameters $t\geq 0$. To introduce the typical object of class $\bX$ we use the concept of Palm distribution introduced in \cite{RotherZaehle}, see also \cite{LastRM}. To apply it, we first notice that $\Sp^2$ is a homogeneous space that can be identified with the quotient $SO(3)/SO(2)$, where the (unimodular) rotation group $SO(3)$ acts transitively on $\Sp^2$ and $SO(2)$ is interpreted as the stabilizer of the north pole $e:=(0,0,1)$. For $x\in\Sp^2$ we let $\Theta_x=\{\vartheta\in SO(3):\vartheta e=x\}$. We let $\nu_e$ be the invariant probability measure on $\Theta_e$ and let $\nu_x:=\nu_e\circ\vartheta^{-1}$ be the image measure for some arbitrary $\vartheta\in\Theta_x$ ($\nu_x$ does not depend on the choice of $\vartheta$). Following \cite{RotherZaehle} we can now define the Palm distribution $\PP_\bX^0$ of the splitting tessellation $Y_t$ with respect to $\bX$ as the probability measure on $(\sT(\Sp^2),\cB(\sT(\Sp^2)))$ given by
$$
\PP_\bX^0(A) := \lambda_\bX^{-1}\,\EE\sum_{x\in\xi_\bX}\int_{\Theta_x}{\bf 1}(\vartheta^{-1} Y_t\in A)\,\nu_x(\dint\!\vartheta)\,,\qquad A\in\cB(\sT(\Sp^2))\,.
$$
Under the Palm distribution $\PP_\bZ^0$ the north pole is the centre of a tessellation cell, which is called the typical cell of $Y_t$ in what follows. Similarly, the typical edge, the typical vertex, the typical side and the typical maximal segment of $Y_t$ are defined as the edge, vertex, side and maximal segment that have the north pole as centre when the splitting tessellation is regarded under the Palm distribution $\PP_\bE^0$, $\PP_\bV^0$, $\PP_\bS^0$ and $\PP_\bM^0$, respectively.

\section{Main results}\label{sec:Results}

\subsection{Mean values}

The purpose of this section is to compute a number of mean values that are associated with a splitting tessellation $Y_t$ for some fixed $t\geq 0$. For this, we only use combinatorial or geometric arguments as well as the martingale property established in Lemma \ref{lem:martingale}. We point out that this self-contained approach has been exploited only partially in the Euclidean case.

Our first result is concerned with the total spherical edge length $L_\bE$ and $L_\bM$ of all edges and all maximal segments, respectively, i.e.,
$$
L_\bE:=L_\bE(t) = 2\pi\sum_{e\in\bE(Y_t)}\sigma_1(e)\qquad\text{and}\qquad L_\bM:=L_\bM(t) = 2\pi\sum_{m\in\bM(Y_t)}\sigma_1(m)\,,
$$
where $\sigma_1$ denotes the normalized spherical length measure. Furthermore, we introduce the total side length of $Y_t$ as
$$
L_\bS:=L_\bS(t) = 2\pi\sum_{p\in\bZ(Y_t)}\sigma_1(\partial p)\,.
$$
The next result shows that in contrast to the Euclidean case the three quantities $L_\bE$, $L_\bM$ and $L_\bS/2$ do not coincide.

\begin{theorem}\label{thm:TotalLength}
Fix $t\geq 0$. Then,
\begin{align*}
L_\bE = 2\pi(1+t-e^{-2t})\,,\qquad L_\bM = 2\pi t\qquad\text{and}\qquad L_\bS=4\pi(t+1)\,.
\end{align*}
\end{theorem}

As above, we consider the time parameter $t\geq 0$ as being fixed and suppress the dependency on $t$ in our notation. So, write $\bX$ for one of the set classes $\bZ,\bE,\bV,\bS$ and $\bM$ introduced in Section \ref{subsec:SetClasses} and recall that by $\lambda_\bX$ we denote the mean number of objects of class $\bX$. 

\begin{theorem}\label{thm:Intensities}
Fix $t\geq 0$. Then $\lambda_\bX$ for $\bX\in\{\bZ,\bE,\bV,\bS,\bM\}$ is given as follows.
\begin{center}
\begin{tabular}{|c||c|c|c|c|c|}
\hline
\parbox[0pt][2em][c]{0cm}{}  & $\bX=\bZ$ & $\bX=\bE$ & $\bX=\bV$ & $\bX=\bS$ & $\bX=\bM$\\
\hline
\parbox[0pt][2em][c]{0cm}{} $\lambda_{\bX}$ & $t^2+2t+2$ & $3(t^2+2t)$ & $2(t^2+2t)$ & $2(2(t^2+2t)+e^{-t})$ & $t^2+2t$\\
\hline
\end{tabular}
\end{center}
\end{theorem}

\bigskip

If $\bX\in\{\bE,\bS,\bM\}$ we let $\ell_\bX$ be the mean spherical length of the typical object of class $\bX$ as introduced in Section \ref{subsec:SetClasses}. Moreover, $\ell_{\partial\bZ}$ and $a_\bZ$ indicate the mean spherical perimeter length and the mean spherical area of the typical cell of $Y_t$. 

\begin{theorem}\label{thm:Lengths}
Fix $t>0$. Then, $a_\bZ={4\pi\over t^2+2t+2}$ and $\ell_\bX$ for $\bX\in\{\partial\bZ,\bE,\bS,\bM\}$ is given as follows.
\begin{center}
\begin{tabular}{|c||c|c|c|c|}
\hline
\parbox[0pt][2em][c]{0cm}{}  & $\bX=\partial\bZ$ & $\bX=\bE$ & $\bX=\bS$ & $\bX=\bM$\\
\hline
\parbox[0pt][2em][c]{0cm}{} $\ell_{\bX}$ & ${4\pi(t+1)\over t^2+2t+2}$ & ${2\pi(1+t-e^{-2t})\over 3(t^2+2t)}$ & ${2\pi(t+1)\over 2(t^2+2t)+e^{-t}}$ & ${2\pi\over t+2}$\\
\hline
\end{tabular}
\end{center}
\end{theorem}

\bigskip

Finally, we consider adjacency relationships. Following the terminology in \cite{WeissCowan}, we let $\bX$ and $\bY$ be two of the set classes introduced in Section \ref{subsec:SetClasses} and say that $x\in\bX$ and $y\in\bY$ are adjacent if $x\subseteq y$ or $y\subseteq x$. By $\mu_{\bX\bY}$ we denote the mean number of objects of class $\bY$ that are adjacent to the typical object of class $\bX$. For example $\mu_{\bZ\bV}$ is the mean number of vertices on the boundary of the typical cell, while $\mu_{\bM\bE}$ is the mean number of edges contained in the typical maximal segment of the splitting tessellation $Y_t$.

\begin{theorem}\label{thn:Adjacencies}
Fix $t>0$ and let $\bX,\bY\in\{\bZ,\bE,\bV,\bS,\bM\}$. Then $\mu_{\bX\bY}$ is given as follows.
\begin{center}
\begin{tabular}{|c||c|c|c|c|c|}
\hline
\parbox[0pt][2em][c]{0cm}{} $\mu_{\bX\bY}$ & $\bX=\bZ$ & $\bX=\bE$ & $\bX=\bV$ & $\bX=\bS$ & $\bX=\bM$\\
\hline
\hline
\parbox[0pt][2em][c]{0cm}{} $\bY=\bZ$ & $1$ & $2$ & $3$ & $-$ & $-$\\
\hline
\parbox[0pt][2em][c]{0cm}{} $\bY=\bE$ & ${6(t^2+2t)\over t^2+2t+2}$ & $1$ & $3$ & ${3(t^2+2t)\over 2(t^2+2t)+e^{-t}}$ & ${3t+2\over t+2}$\\
\hline
\parbox[0pt][2em][c]{0cm}{} $\bY=\bV$ & ${6(t^2+2t)\over t^2+2t+2}$ & $2$ & $1$ & ${5(t^2+2t)\over 2(t^2+2t)+e^{-t}}$ & ${4(t+1)\over t+2}$\\
\hline
\parbox[0pt][2em][c]{0cm}{} $\bY=\bS$ & $-$ & $2$ & $5$ & $-$ & ${4(t+1)\over t+2}$\\
\hline
\parbox[0pt][2em][c]{0cm}{} $\bY=\bM$ & $-$ & ${3t+2\over 3(t+2)}$ & ${2(t+1)\over t+2}$ & ${2t^2+2t\over 2(t^2+2t)+e^{-t}}$ & $1$\\
\hline
\end{tabular}
\end{center}
\end{theorem}

\medskip

\begin{remark}
It appears that the mean adjacencies $\mu_{\bZ\bM}$, $\mu_{\bM\bZ}$, $\mu_{\bS\bS}$, $\mu_{\bZ\bS}$ and $\mu_{\bS\bZ}$ cannot be computed by means of the tools developed so far. This goes hand in hand with the situation for STIT tessellations in the Euclidean plane, where these mean values have only been computed recently in \cite{CowanSTIT,CowanTh} using much more sophisticated tools and arguments.
\end{remark}

Let us rephrase some of the mean adjacencies presented in Theorem \ref{thn:Adjacencies} in a different way and let us compare them with those in the Euclidean case. We first notice that, as $t\to\infty$, the mean values $\mu_{\bX\bY}$ in Theorem \ref{thn:Adjacencies} tend to the corresponding mean values for STIT tessellations in the Euclidean plane, which are all independent of the time parameter $t$ in this case. However, they behave substantially different for `small' values of $t$. We illustrate this phenomenon in three representative cases. At first, the mean number of vertices on the boundary of the typical cell is ${6(t^2+2t)\over t^2+2t+2}$. As $t\to\infty$, this tends to $6$, which shows that for large $t$, the typical cell has six vertices on its boundary.  This is similar to the Euclidean case, where this mean values does not depend on $t$. On the other hand, for small time parameters $t$, the mean values on the sphere and in the plane behave differently. While on the sphere $\mu_{\bZ\bV}\to 0$, as $t\to 0$, we have that $\mu_{\bZ\bV}=6$ in the Euclidean case, independently of $t$ and as anticipated above. Secondly, we consider the typical maximal segment. The mean number of vertices in its relative interior is given by $\mu_{\bM\bV}-2={2t\over t+2}$. This tends to $2$, as $t\to\infty$, which is the mean value known from the Euclidean case (where it is independent of $t$). Similarly, the mean number of vertices in the relative interior of the typical side is $\mu_{\bS\bV}-2={t^2+2t-2e^{-t}\over 2(t^2+2t)+e^{-t}}$, which tends to $1/2$, as $t\to\infty$, which in turn is also the corresponding value in the Euclidean case (again, independently of $t$). On the contrary, we have that $\mu_{\bM\bV}-2\to 0$ and $\mu_{\bS\bV}\to 0$, as $t\to 0$.

\medbreak

Let us finally compare some of the mean values we computed with those for Poisson great circle tessellations on the sphere, see \cite{ArbeiterZaehle,MilesSphere}. To define the model, let for some $\gamma>0$, $\eta_\gamma$ be a Poisson point process on $\Sp^2$ whose intensity measure is given by $\gamma\,\sigma_2$. By the Poisson great circle tessellation on $\Sp^2$ with parameter $\gamma>0$ we understand the random closed set
$$
Y_\gamma^{\rm GC} := \{A\}\cup\bigcup_{u\in\eta_\gamma}(\Sp^2\cap u^\perp)\,.
$$
To have the analogy with our splitting tessellations, we have included the equator in $Y_\gamma^{\rm GC}$. Note that in contrast to splitting tessellations the cells of a Poisson great circle tessellation are side-to-side, meaning that the intersection of two adjacent cells is either a common corner ($0$-face) or a common side ($1$-face) of both cells. To compare a great circle tessellation with the splitting tessellation $Y_t$, we choose $\gamma=t\geq 0$. This turns out to be a reasonable choice, as the following proposition shows.

\begin{proposition}\label{prop:GreatCircles}
For all $t\geq 0$ one has that
\begin{alignat*}{3}
& L_\bE(Y_t^{\rm GC}) = 2\pi(1+t-e^{-t})\,,\qquad && L_{\bS}(Y_t^{\rm GC}) = 4\pi(t+1)\,,\qquad && \phantom{xxxxx}\\
&\lambda_\bV(Y_t^{\rm GC}) = t^2+2t\,,\qquad &&\lambda_{\bE}(Y_t^{\rm GC}) = 2(t^2+2t)\,,\qquad &&\lambda_{\bZ}(Y_t^{\rm GC}) = t^2+2t+2\,.
\end{alignat*}
Moreover, for $t>0$,
\begin{alignat*}{3}
& \ell_\bE(Y_t^{\rm GC}) = {\pi(1+t-e^{-t})\over t^2+2t}\,,\qquad && \ell_{\partial\bZ}(Y_t^{\rm GC}) = {4\pi(t+1)\over t^2+2t+2}\,,\qquad && a_\bZ(Y_t^{\rm GC}) = {4\pi\over t^2+2t+2}\,.
\end{alignat*}
\end{proposition}

A comparison of Proposition \ref{prop:GreatCircles} and Theorem \ref{thm:Lengths} shows, for example, that the mean area and the mean perimeter length of the typical cell of $Y_\gamma^{\rm GC}$ and $Y_t$ coincide, while Theorem \ref{thm:TotalLength} ensures that $L_\bS(Y_t)=L_\bS(Y_t^{\rm GC})$. On the other hand, a Poisson great circle tessellation cannot have maximal segments. 

\subsection{Capacity functional and intersection property}\label{subsec:CapaIntersection}

In this section we turn to more sophisticated properties of the splitting tessellations $Y_t$. Let us fix $t\geq 0$ and recall from \cite[Chapter 2.2]{SchneiderWeil} that the so-called capacity functional of the random closed set $Y_t$ on $\Sp^2$ is defined as
$$
T_{Y_t}(C)=\PP(Y_t\cap C\neq\emptyset)\,,
$$
where $C$ is a closed subset of $\Sp^2$. It uniquely characterizes the distribution of $Y_t$ and is the analogue of a distribution function of a real-valued random variable, cf.\ \cite{SchneiderWeil}. For this reason, it is one of the fundamental characteristics that can be associated with a random closed set. Our first goal is to compute the capacity functional of $Y_t$. Not surprisingly, the result is similar to what is known from the Euclidean case, see Lemma 3 and Lemma 4 in \cite{NagelWeiss2005}.

To present the theorem, we denote the spherical convex hull of a set $B\subset\Sp^2$ by $\conv(B)$ and write $[B_1|B_2]$ for the set of great circles that separate two sets $B_1,B_2\subset\Sp^2$ that are contained in a common open hemisphere, i.e.,
$$
[B_1|B_2]=\{g\in\sG(\Sp^2):B_1\cap g=B_2\cap g=\emptyset,\conv(B_1\cup B_2)\cap g\neq\emptyset\}\,.
$$

\begin{theorem}\label{thm:Capacity}
Let $C\subset\Sp^2$ be closed and contained either in $\Sp^2_+\setminus A$ or $\Sp^2_-\setminus A$. Suppose in addition that $C$ is connected. Then
$$1-T_{Y_t}(C)=e^{-t\tau([C])}\,.
$$
More generally, if $C=C_1\cup\ldots\cup C_m$ for some $m\geq 2$ with pairwise disjoint, connected and closed subsets $C_1,\ldots,C_m\subset\Sp^2$ with either $C_1,\ldots,C_m\subset\Sp_+^2\setminus A$ or $C_1,\ldots,C_m\subset\Sp_-^2\setminus A$ for all $i\in\{1,\ldots,m\}$, then
\begin{equation}\label{eq:CapaMulti}
\begin{split}
1-T_{Y_t}(C) = & e^{-t\tau([\conv(C)])} + \sum_{Z_1,Z_2}\tau([Z_1|Z_2])\\
&\qquad\times\int_0^te^{-s\tau([\conv(C)])}\,(1-T_{Y_{t-s}}(Z_1))(1-T_{Y_{t-s}}(Z_2))\,\dint\!s
\end{split}
\end{equation}
where the sum runs over all partitions $Z_1=\bigcup_{j\in J}C_j$, $Z_2=\bigcup_{j\in J^c}C_j$, where $J$ is a proper non-empty subset of $\{1,\ldots,m\}$.
\end{theorem}

It is often important in the theory of spherical random tessellations that one can control the point process that arises as the intersection of a spherical tessellation with a fixed great circle. While this is easy for Poisson great circle tessellations in view of the well-known mapping properties of Poisson point processes, the intersection point process is only poorly understood for Poisson-Voronoi tessellations on the sphere (and even in the Euclidean case). We use the previous result on the capacity functional of $Y_t$ to prove that the point process that arises if the splitting tessellation $Y_t$ is intersected by a fixed great circle different from $A$ is the concatenation of two independent Poisson point processes plus a pair of antipodal points on the equator $A$. This property is similar to the Euclidean case, but on the sphere the equator plays a special role.

\begin{theorem}\label{thm:intersection}
Fix $t\geq 0$ and let $g\in\sG(\Sp^2)\setminus\{A\}$ be a fixed great circle. Then the random point processes $Y_t\cap(\Sp_+^2\setminus A)\cap g$ and $Y_t\cap(\Sp_-^2\setminus A)\cap g$ are independent and both are Poisson point processes whose intensity measure is given by $t$ times the normalized spherical length measure on $(\Sp_+^2\setminus A)\cap g$ and $(\Sp_-^2\setminus A)\cap g$, respectively.
\end{theorem}

\begin{remark}\label{rem:InterEquator}
In Theorem \ref{thm:intersection} one cannot replace the fixed great circle $g$ by a random one that depends on $Y_t$, for example by a great circle that carries one of the maximal segments of $Y_t$.

If $g$ is the equator $A$ then the point process induced by the random process $(Y_t)_{t>0}$ in the upper and the lower hemisphere is a Poisson point process plus a pair of two antipodal point and both point processes are independent. Hence, the induced point process on $A$ is the superposition of two independent Poisson point processes plus two pairs of antipodal points that arise as the endpoints of the two maximal segments that appear first in the two hemispheres (provided they exist).
\end{remark}

\section{Proofs}\label{sec:Proofs}

\subsection{Proof of Theorem \ref{thm:TotalLength}}

We start with the formula for $L_\bM$. Since the equator $A$ is, by definition, not a maximal segment, we can re-write $L_\bM$ as
\begin{align}\label{eq:RepresentationL1}
L_\bM = \pi\sum_{p\in\bZ(Y_t)} \sigma_1(\partial p)-2\pi\,.
\end{align}
We now apply Lemma \ref{lem:martingale} with $\phi(p)=\sigma_1(\partial p)$, take expectations in \eqref{eq:Martingale} and use Fubini's theorem, to see that
\begin{align*}
\EE\sum\limits_{p\in\bZ(Y_t)}\sigma_1(\partial p) &= \EE\Sigma_\phi(Y_0) + \EE\int_0^t\mathbb{L}\Sigma_\phi(Y_s)\,\dint\!s\\
&= 2 + \int_0^t\EE\sum_{p\in\bZ(Y_s)}\int_{[p]}[\Sigma_\phi(\oslash_{p,g}(Y_s))-\Sigma_\phi(Y_s)]\,\tau(\dint\!g)\dint\!s\\
&= 2+2\int_0^t\EE\sum_{p\in\bZ(Y_s)}\int_{[p]}\sigma_1(p\cap g)\,\tau(\dint\!g)\dint\! s\\
&= 2+2\int_0^t\EE\sum_{p\in\bZ(Y_s)}\int_{[p]}v_1(p\cap g)\,\tau(\dint\!g)\dint\! s\,.
\end{align*}
To the inner integral we apply the spherical Crofton formula from Lemma \ref{lem:Crofton}, which leads to
\begin{align*}
\EE\sum\limits_{p\in\bZ(Y_t)}\sigma_1(\partial p) = 2 + 2\int_0^t\EE\sum_{p\in\bZ(Y_s)}v_2(p)\,\dint\!s\,.
\end{align*}
Since $Y_t$ is a spherical tessellation, $\sum_{p\in\bZ(Y_s)}v_2(p)=v_2(\Sp^2)=1$ almost surely and hence
\begin{align}\label{eq:XXXX}
\EE\sum\limits_{p\in\bZ(Y_t)}\sigma_1(\partial p) = 2 + 2t\,.
\end{align}
Together with \eqref{eq:RepresentationL1} this proves that $L_\bM=\pi(2+2t)-2\pi=2\pi t$.

To deal with the total spherical edge length $L_\bE$ we observe that $L_\bE=L_\bM+2\pi$ whenever at least one of the two cells $\Sp^2_\pm$ of the initial tessellation $Y_0$ has been split by a great circle until time $t$. The latter event has probability $(1-e^{-t})+(1-e^{-t})-(1-e^{-t})(1-e^{-t})=1-e^{-2t}$, since the two cells $\Sp^2_\pm$ have independent and exponentially distributed random lifetimes with parameter $1$. This shows that $L_\bE=L_\bM+2\pi(1-e^{-2t})=2\pi(1+t-e^{-2t})$. To complete the proof, we finally notice that $L_\bS=2(L_\bM+2\pi)=4\pi t+4\pi$ since $Y_0=\{A,A\}$ and the equator has length $2\pi$.\hfill $\Box$

\subsection{Proof of Theorem \ref{thm:Intensities}}

We start with the number of cells and choose $\phi(p)\equiv 1$ in Lemma \ref{lem:martingale}. Upon taking expectations in \eqref{eq:Martingale} this shows that
\begin{align*}
\lambda_\bZ &= \EE\sum_{p\in\bZ(Y_t)} 1 = \EE\Sigma_\phi(Y_0) + \int_0^t \EE \sum_{p\in\bZ(Y_s)}\int_{[p]} [\Sigma_\phi(\oslash_{p,g}(Y_s))-\Sigma_\phi(Y_s)]\,\tau(\dint\!g)\dint\!s\\
&=2+\int_0^t\EE \sum_{p\in\bZ(Y_s)}\tau([p])\,\dint\!s = 2+\int_0^t\EE\sum_{p\in\bZ(Y_s)}\sigma_1(\partial p)\,\dint\!s\,.
\end{align*}
We know from \eqref{eq:XXXX} that the inner expectation equals $2(s+1)$, which implies that
\begin{align*}
\lambda_\bZ = 2+\int_0^t 2(s+1)\,\dint\!s = t^2+2t+2\,.
\end{align*}
Moreover, since each cell split increases the number of cells and the number of maximal segments by $1$, we have the obvious relation $\lambda_\bM=\lambda_\bZ-2=t^2+2t$.

Next, we notice that each cell split increases the number of vertices by $2$ and the number of edges by $3$. Moreover, we also notice that at time zero, there are no edges and vertices. This implies that $\lambda_\bV=2(\lambda_\bZ-2)=2(t^2+2t)$ and $\lambda_\bE=3(\lambda_\bZ-2)=3(t^2+2t)$.

It remains to compute $\lambda_\bS$, the mean number of sides of $Y_t$. To determine this quantity, we let $\cS(p)$ be the collection of sides of a spherically convex polygon $p\in\sP(\Sp^2)$ and let $T_+$ and $T_-$ be the random lifetimes of the two cells $\Sp^2_+$ and $\Sp^2_-$ of the initial tessellation $Y_0$. Now, we observe that $\lambda_\bS$ satisfies the relation
\begin{align*}
\lambda_\bS &= \EE\sum_{p\in\bZ,p\subseteq\Sp^2_+}|\cS(p)| + \EE\sum_{p\in\bZ,p\subseteq\Sp^2_-}|\cS(p)|\\
& = \EE\Big[1+\Big(4\Big(\sum_{m\in\bM,m\subset\Sp^2_+}1\Big)-1\Big){\bf 1}(T_+\leq t)\Big]+\EE\Big[1+\Big(4\Big(\sum_{m\in\bM,m\subset\Sp^2_-}1\Big)-1\Big){\bf 1}(T_-\leq t)\Big]\,.
\end{align*}
Since $T_+$ and $T_-$ are exponentially distributed with parameter $1$ and since
$$
\EE\Big[ \Big(\sum_{m\in\bM,m\subset\Sp^2_+}1\Big){\bf 1}(T_+\leq t)\Big] = \EE\Big[ \Big(\sum_{m\in\bM,m\subset\Sp^2_-}1\Big){\bf 1}(T_-\leq t)\Big] = {1\over 2}\lambda_\bM = {1\over 2}(t^2+2t)\,,
$$
we conclude that
\begin{align*}
\lambda_\bS= 2\Big(1+ 4\Big({1\over 2}(t^2+2t)\Big)-(1-e^{-t})\Big) = 4(t^2+2t)+2e^{-t}\,.
\end{align*}
This completes the proof. \hfill $\Box$

\subsection{Proof of Theorem \ref{thm:Lengths}}

We have that
$$
\ell_{\partial\bZ}={L_\bS\over\lambda_\bZ}\,,\qquad\ell_{\bE} = {L_\bE\over\lambda_\bE}\,,\qquad\ell_\bS = {L_\bS\over\lambda_\bS} \qquad\text{and}\qquad \ell_\bM = {L_\bM\over\lambda_\bM}\,.
$$
Moreover, $a_\bZ$ satisfies the relation $a_\bZ=4\pi/\lambda_\bZ$. Now, Theorem \ref{thm:TotalLength} and Theorem \ref{thm:Intensities} complete the proof. \hfill $\Box$

\subsection{Proof of Theorem \ref{thn:Adjacencies}}

We prepare the proof with the following lemma, which is applied repeatedly below.

\begin{lemma}\label{lem:KnotenAufAequator}
The mean number of vertices of $Y_t$ that are located on the equator $A$ equals $4t$ and the mean number of sides of $Y_t$ on $A$ is $4t+2e^{-t}$, i.e.,
$$
\EE\sum_{v\in\bV}{\bf 1}(v\in A) = 4t\qquad\text{and}\qquad\EE\sum_{s\in\bS}{\bf 1}(s\subset A)=4t+2e^{-t}\,.
$$
\end{lemma}
\begin{proof}
Let $T_\pm$ be the random lifetimes of the two hemispheres $\mathbb{S}_\pm^2$ bounded by $A$. If $T_+>t$ and $T_->t$ then, at time $t$, the number of vertices on $A$ is zero, while $Y_t$ has two cells. If exactly one of the two lifetimes $T_\pm$ exceeds $t$ (i.e., $T_+>t$ and $T_-\leq t$, or $T_+\leq t$ and $T_->t$), then the difference between the number of vertices on $A$ and the number of cells of $Y_t$ that have non-empty intersection with $A$ is $1$, while if $T_+\leq t$ and $T_-\leq t$ the difference is zero. Thus, almost surely,
$$
\sum_{v\in\bV}{\bf 1}(v\in A) = \Big[\sum_{p\in\bZ}{\bf 1}(p\cap A\neq\emptyset)\Big]-{\bf 1}(T_+>t)-{\bf 1}(T_->t)\,.
$$
Taking expectations we see that
\begin{align}\label{eq:Equator1}
\EE\sum_{v\in\bV}{\bf 1}(v\in A) = \Big[\EE\sum_{p\in\bZ}{\bf 1}(p\cap A\neq\emptyset)\Big]-2e^{-t}\,,
\end{align}
since $T_+$ and $T_-$ are independent and exponentially distributed random variables with mean $1$. To determine the expectation in brackets, we use that
$$
\EE\sum_{p\in\bZ}{\bf 1}(p\cap A\neq\emptyset) = 2\,\EE\sum_{p\in\bZ(Y_t\cap\Sp_+^2)}{\bf 1}(p\cap A\neq\emptyset)
$$
and apply the martingale property of \eqref{eq:Martingale} with $\phi(p)={\bf 1}(p\subset\Sp^2_+,p\cap A\neq\emptyset)$ there to see that
\begin{align*}
& \EE\sum_{p\in\bZ(Y_t\cap\Sp_+^2)}{\bf 1}(p\cap A\neq\emptyset) \\
&= \EE\Sigma_\phi(Y_0)+\int_0^t\EE\sum_{p\in\bZ(Y_s)}\int_{[p]}[\Sigma_\phi(\oslash_{p,g}(Y_s))-\Sigma_\phi(Y_s)]\,\tau(\dint\!g)\dint\!s\\
&= 1 + \int_0^t\EE\sum_{p\in\bZ(Y_s\cap\Sp_+^2)}\int_{[p]}{\bf 1}(p\cap A\cap g\neq\emptyset)\,\tau(\dint\! g)\dint\!s\\
&=1+\int_0^t\EE\sum_{p\in\bZ(Y_s\cap\Sp_+^2)}\tau([p\cap A])\,\big({\bf 1}(T_+>s)+{\bf 1}(T_+\leq s)\big)\,\dint\!s\,.
\end{align*}
If $T_+>s$, we have that $\tau([p\cap A])=1$ for the single cell $p=\Sp_+^2\in\bZ(Y_s\cap\Sp_+^2)$. If $T_+\leq s$ then $\tau([p\cap A])=2v_1(p\cap A)$ for all $p\in\bZ(Y_s\cap\Sp_+^2)$. Since $\PP(T_+>s)=e^{-s}$ and since the functional $v_1$ is additive and satisfies $v_1(A)=1$, this leads to
\begin{align*}
\EE\sum_{p\in\bZ(Y_t\in\Sp_+^2)}{\bf 1}(p\cap A\neq\emptyset) &= 1 + \int_0^te^{-s}+2(1-e^{-s})\EE\sum_{p\in\bZ(Y_s\cap\Sp_+^2)}v_1(p\cap A)\,\dint\!s\\
&=1 + \int_0^te^{-s}+2(1-e^{-s})\,\dint\!s = 2t+e^{-t}\,,
\end{align*}
whence
\begin{equation}\label{eq:SidesEquator}
\EE\sum_{p\in\bZ}{\bf 1}(p\cap A\neq\emptyset) = 2\,\EE\sum_{p\in\bZ(Y_t\cap\Sp_+^2)}{\bf 1}(p\cap A\neq\emptyset) = 4t+2e^{-t}\,,
\end{equation}
which proves the claim about the sides on the equator. In conjunction with \eqref{eq:Equator1} we also conclude that
$$
\EE\sum_{v\in\bV}{\bf 1}(v\in A) = (4t+2e^{-t})-2e^{-t} = 4t
$$
and this completes the proof of the lemma.
\end{proof}

\begin{remark}
An alternative proof of Lemma \ref{lem:KnotenAufAequator} can be given by means of Theorem \ref{thm:intersection} in conjunction with Remark \ref{rem:InterEquator}. We have decided to embed the result into our self-contained approach, which only relies on the martingale property presented in Lemma \ref{lem:martingale}.
\end{remark}

\begin{proof}[Proof of Theorem \ref{thn:Adjacencies} with $\bX=\bZ$]
Clearly, $\mu_{\bZ\bZ}=1$. Next, we notice that
$$
\mu_{\bZ\bV}={3\lambda_\bV\over\lambda_\bZ}\,,\qquad\text{and}\qquad \mu_{\bZ\bE} = \mu_{\bZ\bV}\,,
$$
since each vertex belongs to precisely $3$ cells. Applying Theorem \ref{thm:Intensities} leads to the precise values.
\end{proof}

\begin{proof}[Proof of Theorem \ref{thn:Adjacencies} with $\bX=\bE$]
Trivially, one has that $\mu_{\bE\bE}=1$. Moreover, since each edge is contained in the boundary of exactly two cells, we have that $\mu_{\bE\bZ}=\mu_{\bE\bS}=2$ and since each edge has exactly two endpoints, we find that $\mu_{\bE\bV}=2$. Finally, each edge is contained in exactly one maximal segment, except in the case that it lies on the equator (which in turn is not a maximal segment). In view of Lemma \ref{lem:KnotenAufAequator} we know that the mean number of edges on $A$ is $4t$ and hence $\mu_{\bE\bM}=(\lambda_\bE-4t)/\lambda_\bE$. An application of Theorem \ref{thm:Intensities} completes the proof in this case.
\end{proof}

\begin{proof}[Proof of Theorem \ref{thn:Adjacencies} with $\bX=\bV$]
Clearly, $\mu_{\bV\bV}=1$. Each vertex is contained in the boundary of precisely $3$ cells and is endpoint of precisely $3$ edges and $5$ sides. This implies $\mu_{\bV\bZ}=\mu_{\bV\bE}=3$ and $\mu_{\bV\bS}=5$. To compute $\mu_{\bV\bM}$ we observe that each vertex is the endpoint of exactly one maximal segment and is contained in the relative interior of exactly one maximal segment, except of the case that the vertex is located on the equator, in which case it is only the endpoint of one maximal segment. This leads to the identity
$$
\mu_{\bV\bM}={2(\lambda_\bV-4t)\over \lambda_\bV}+{4t\over \lambda_\bV}
$$
and an application of Theorem \ref{thm:Intensities} yields the precise value.
\end{proof}

\begin{proof}[Proof of Theorem \ref{thn:Adjacencies} with $\bX=\bS$]
We have that $\mu_{\bS\bE}=2\lambda_\bE/\lambda_\bS$ and $\mu_{\bS\bV}=5\lambda_\bV/\lambda_\bS$, since each edge belongs to precisely two sides and since each vertex is the endpoint of precisely $4$ sides and is contained in the relative interior of a further side. To compute $\mu_{\bS\bM}$ we notice that each side is contained in exactly one maximal segment, the only exception being sides on the equator, which are contained in no maximal segment. Lemma \ref{lem:KnotenAufAequator} implies that the mean number of sides of $Y_t$ on $A$ equals $4t+2e^{-t}$. This gives the relation
$$
\mu_{\bS\bM}={\lambda_\bS-4t-2e^{-t}\over\lambda_\bS}\,.
$$
Now, application of Theorem \ref{thm:Intensities} completes the proof.
\end{proof}

\begin{proof}[Proof of Theorem \ref{thn:Adjacencies} with $\bX=\bM$]
Clearly, $\mu_{\bM\bM}=1$. Moreover, using Lemma \ref{lem:KnotenAufAequator} and recalling that the equator is not a maximal segment, we find that
$$
\mu_{\bM\bE}={\lambda_\bE-4t\over\lambda_\bM}\,,\qquad \mu_{\bM\bV} = {2\lambda_\bV-4t\over\lambda_\bM}\qquad\text{and}\qquad\mu_{\bM\bS} = {\lambda_\bS-4t-2e^{-t}\over\lambda_\bM}
$$
and Theorem \ref{thm:Intensities} completes the proof of Theorem \ref{thn:Adjacencies}.
\end{proof}

\subsection{Proof of Proposition \ref{prop:GreatCircles}}

Let $N$ be a Poisson random variable with parameter $t\geq 0$. If $N=0$, the great circle tessellation has two cells $\Sp_\pm^2$ that are bounded by the equator, no edge and no vertex. If otherwise $N\geq 1$, we have $N(N+1)$ vertices, $2N(N+1)$ edges and $N^2+N+2$ cells as one easily checks by induction (see Equation (6.1) in \cite{MilesSphere}). Thus,
\begin{align*}
L_\bE(Y_t^{\rm GC}) & = \sum_{k=1}^\infty 2\pi(k+1)\,\PP(N=k) = 2\pi(1+t-e^{-t})\,,\\
L_\bS(Y_t^{\rm GC}) & = 4\pi\,\PP(N=0) +\sum_{k=1}^\infty 4\pi(k+1)\,\PP(N=k) = 4\pi(t+1)
\end{align*}
and
\begin{align*}
\lambda_\bV(Y_t^{\rm GC}) &= \sum_{k=1}^\infty k(k+1)\,\PP(N=k) = t^2+2t\,,\\
\lambda_\bE(Y_t^{\rm GC}) &= \sum_{k=1}^\infty 2k(k+1)\,\PP(N=k) = 2(t^2+2t)\,,\\
\lambda_\bZ(Y_t^{\rm GC}) &= 2\,\PP(N=0)+\sum_{k=1}^\infty (k^2+k+2)\,\PP(N=k) = t^2+2t+2\,.\\
\end{align*}
Moreover, for $a_\bZ(Y_t^{\rm GC})$, $\ell_{\partial\bZ}(Y_t^{\rm GC})$ and $\ell_\bE(Y_t^{\rm GC})$ we have that
$$
a_\bZ(Y_t^{\rm GC}) = {4\pi\over\lambda_\bZ(Y_t^{\rm GC})} = {4\pi\over t^2+2t+2}\,,\qquad\ell_{\partial\bZ}(Y_t^{\rm GC}) = {L_\bS(Y_t^{\rm GC})\over\lambda_\bZ(Y_t^{\rm GC})} = {4\pi(t+1)\over t^2+2t+2}
$$
and $\ell_\bE(Y_t^{\rm GC})={L_{\bE}(Y_t^{\rm GC})\over\lambda_\bE(Y_t^{\rm GC})}={\pi(1+t-e^{-t})\over t^2+2t}$. This completes the proof.\hfill $\Box$

\subsection{Proof of Theorem \ref{thm:Capacity}}

Let us first suppose that $C$ is connected and observe that
$$
\PP(Y_t\cap C=\emptyset) = \EE\sum_{p\in\bZ}{\bf 1}(C\subset p)\,.
$$
It then follows from Lemma \ref{lem:martingale} that the stochastic process defined there with the special choice $\phi(p):={\bf 1}(p\subset C)$ is a martingale. Upon taking expectations in \eqref{eq:Martingale} this yields
\begin{align*}
\PP(Y_t\cap C=\emptyset) = \EE\Sigma_\phi(Y_t) &= 
\EE\Sigma_\phi(Y_0)+\int_0^t \EE\sum\limits_{p\in\bZ(Y_s)}\int_{[p]}[\Sigma_\phi(\oslash_{p,g}(Y_s))-\Sigma_\phi(Y_s)]\,\tau(\dint\!g)\,\mathrm{d}s\\
&=1-\int_0^t \EE\sum_{p\in\bZ(Y_s)}{\bf 1}(C\subset p) \int_{[p]}{\bf 1}(C\cap g\neq\emptyset)\,\tau(\dint\!g)\,\mathrm{d}s\\
&=1-\tau([C])\int_0^t\EE\sum_{p\in\bZ(Y_s)}{\bf 1}(C\subset p)\,\mathrm{d}s\\
&=1-\tau([C])\int_{0}^t\mathbb{P}(Y_s\cap C=\emptyset)\,\mathrm{d}s\,.
\end{align*}
This leads to the integral equation
\begin{align*}
y(t)=1-\tau([C]) \int_0^ty(s)\,\mathrm{d}s\qquad\text{with}\qquad y(0)=1
\end{align*}
for $y(t):=\PP(Y_t\cap C=\emptyset)=1-T_{Y_t}(C)$, which has the unique solution $y(t)=e^{-t\tau([C])}$.

We assume now that $C$ has more than one connected component and partition the event that $Y_t\cap C=\emptyset$ as follows:
\begin{align*}
1-T_{Y_t}(C) &=\PP(Y_t\cap C=\emptyset) \\
&= \PP(Y_t\cap C=\emptyset,Y_t\cap\conv(C)=\emptyset)+\PP(Y_t\cap C=\emptyset,Y_t\cap\conv(C)\neq\emptyset)\,.
\end{align*}
Since $\PP(Y_t\cap C=\emptyset,Y_t\cap\conv(C)=\emptyset) = \PP(Y_t\cap\conv(C)=\emptyset)$ and since $\conv(C)$ is connected, we have from the first part of the proof that
$$
1-T_{Y_t}(C) = e^{-t\tau([\conv(C)])} + \PP(Y_t\cap C=\emptyset,Y_t\cap\conv(C)\neq\emptyset)\,.
$$
If $Y_t\cap C=\emptyset$ and $Y_t\cap\conv(C)\neq\emptyset$, there exists a time $s\in(0,t)$ at which the connected components of $C$ are separated for the first time by a great circle (different from $A$) into two parts $Z_1=\bigcup_{j\in J}C_j$, $Z_2=\bigcup_{j\in J^c}C_j$, where $J$ is a proper non-empty subset of $\{1,\ldots,m\}$. Moreover, after this separation, the sets $Z_1$ and $Z_2$ are not intersected by spherical line segments that are constructed in the remaining time interval $(s-t)$. Using the Markovian description of the continuous time dynamic of $(Y_t)_{t>0}$ this means that
\begin{align*}
& \PP(Y_t\cap C=\emptyset,Y_t\cap\conv(C)\neq\emptyset)\\
&= \sum_{Z_1,Z_2}\tau([Z_1|Z_2])\int_0^t\PP(Y_s\cap\conv(C)=\emptyset)\PP(Y_{t-s}\cap Z_1=\emptyset)\PP(Y_{t-s}\cap Z_2=\emptyset)\,\dint\!s\\
&= \sum_{Z_1,Z_2}\tau([Z_1|Z_2])\int_0^te^{-s\tau([\conv(C)])}\,(1-T_{Y_{t-s}}(Z_1))(1-T_{Y_{t-s}}(Z_2))\,\dint\!s\,.
\end{align*}
This completes the proof of the theorem.\hfill $\Box$

\subsection{Proof of Theorem \ref{thm:intersection}}

That the point processes $Y_t\cap\Sp_+^2\cap g$ and $Y_t\cap\Sp_-^2\cap g$ are independent is clear from the Markovian construction of $Y_t$. So, it is sufficient to show that $Y_t\cap(\Sp_+^2\setminus A)\cap g$ is a Poisson point process with the correct intensity measure. It is well known from the theory of random sets (see \cite[Theorem 3.6.3]{SchneiderWeil}) that for this it is enough to prove that for all $m\in\{1,2,3,\ldots\}$ and all disjoint spherical line segments (intervals) $I_1,\ldots,I_m\subset(\Sp_+^2\setminus A)\cap g$ one has that
\begin{equation}\label{eq:IntersectionToShow}
\PP(Y_t\cap I_1=\emptyset,\ldots,Y_t\cap I_m=\emptyset)=\prod_{j=1}^me^{-t\tau([I_j])}\,.
\end{equation}

Without loss of generality we can and will assume that the intervals $I_1,\ldots,I_m$ are ordered in such a way that first interval $I_1$, then interval $I_2$ etc.\ is visited when travelling along $\Sp_+^2\cap g$ from one endpoint of $A\cap g$ to the other (the choice of the orientation is irrelevant).

To establish \eqref{eq:IntersectionToShow} we proceed by induction on $m$. For $m=1$ this just follows from the formula for the capacity functional of $Y_t$ for connected argument sets in Theorem \ref{thm:Capacity}. So, let us assume that \eqref{eq:IntersectionToShow} is valid for $m-1\geq 2$ there instead of $m$. We use \eqref{eq:CapaMulti} with $C=I_1\cup\ldots\cup I_m$ and determine the terms appearing there. Since both $Z_1$ and $Z_2$ consist of less than $m$ disjoint spherical intervals, we can apply our induction hypothesis to deduce that
\begin{align*}
1-T_{Y_{t-s}}(Z_1) &= \PP(Y_{t-s}\cap Z_1=\emptyset) = \prod_{j\in J}e^{-(t-s)\tau([I_j])}\,,\\
1-T_{Y_{t-s}}(Z_2) &=\PP(Y_{t-s}\cap Z_2=\emptyset) = \prod_{j\in J^c}e^{-(t-s)\tau([I_j])}
\end{align*}
and hence
\begin{equation*}
(1-T_{Y_{t-s}}(Z_1))(1-T_{Y_{t-s}}(Z_2))= \prod_{j=1}^me^{-(t-s)\tau([I_j])}\,.
\end{equation*}
Next, we denote by $I_{j,j+1}$ the spherical interval on $g$ that is in between $I_j$ and $I_{j+1}$ for $j\in\{1,\ldots,m-1\}$, and observe that $\tau([I])=\tau([I_1])+\ldots+\tau([I_m])+\tau([I_{1,2}])+\ldots+\tau([
I_{m-1,m}])$, where $I:={\rm conv}(C)$. This implies that
\begin{align*}
e^{-s\tau([I])}\,(1-T_{Y_{t-s}}(Z_1))(1-T_{Y_{t-s}}(Z_2))= e^{-s\sum_{j=1}^{m-1}\tau([I_{j,j+1}])}\,e^{-t\sum_{j=1}^{m}\tau([I_{j}])}
\end{align*}
so that the integral in \eqref{eq:CapaMulti} equals
\begin{align*}
{1\over \sum_{j=1}^{m-1}\tau([I_{j,j+1}])}\,(1-e^{-t\sum_{j=1}^{m-1}\tau([I_{j,j+1}])})\,e^{-t\sum_{j=1}^{m}\tau([I_{j}])}\,.
\end{align*}
What remains is to compute the value of the sum over $Z_1,Z_2$ in \eqref{eq:CapaMulti}. We have that $\tau([Z_1|Z_2]) = 0$ if and only if $ {\rm conv}(Z_1)\cap{\rm conv}(Z_2) \neq \emptyset$. Thus, in order to have $\tau([Z_1|Z_2])\neq 0$, the two sets $Z_1$ and $Z_2$ must be separated by exactly one of the spherical intervals $I_{i,i+1}$. This immediately implies the relation
$$
\sum_{Z_1,Z_2}\tau([Z_1|Z_2]) = \sum_{j=1}^{m-1}\tau([I_{j,j+1}])\,.
$$
Combining these facts with \eqref{eq:CapaMulti} and the formula for $1-T_{Y_t}(I)$ that follows from the first part of Theorem \ref{thm:Capacity}, we arrive at
\begin{align*}
1-T_{Y_t}(I) &= \PP(Y_t\cap I_1=\emptyset,\ldots,Y_t\cap I_m=\emptyset)\\
&= e^{-t\big(\sum_{j=1}^{m}\tau([I_{j}])+\sum_{j=1}^{m-1}\tau([I_{j,j+1}])\big)}+(1-e^{-t\sum_{j=1}^{m-1}\tau([I_{j,j+1}])})\,e^{-t\sum_{j=1}^{m}\tau([I_{j}])}\\
&=e^{-t\sum_{j=1}^{m}\tau([I_{j}])} = \prod_{j=1}^m e^{-t\tau([I_j])}\,.
\end{align*}
In view of \eqref{eq:IntersectionToShow} this completes the proof.\hfill $\Box$

\begin{acknowledgements}
We would like to thank Eva-Maria Gassner for producing the photographs shown in Figure 1.\\ CD and CT acknowledge the support of SFB-TR 12. JH has been supported by GRK 2131.
\end{acknowledgements}

\bibliography{STITSphere}

\end{document}